\documentclass[a4paper,11pt]{article}
\newcommand{\beq}{\begin{equation}}
\newcommand{\eeq}{\end{equation}}

\newcommand{\tab}{\hspace*{2em}}
\newcommand{\fNorm}[1]{\lVert#1\rVert}

\usepackage{amsthm,amssymb,amsmath}
\usepackage{graphicx}
\usepackage{epstopdf}
\usepackage{mathdots}
\usepackage[lmargin=.75in,rmargin=.75in,tmargin=.75in,bmargin=.75in]{geometry}
\usepackage{float}
\usepackage{hyperref}
\usepackage{subfigure}
\newtheorem{thm}{Theorem}[section]

\theoremstyle{remark}

\theoremstyle{definition}

\begin{document}

\title{The application of the exact operational matrices for solving\\
 the Emden-Fowler equations, arising in astrophysics}
\author {K. Parand $^{\mbox{\footnotesize }}$\thanks{E-mail
address: \texttt{k\_parand@sbu.ac.ir}, Corresponding author, (K. Parand)}\hspace{1mm}
, \normalsize{Sayyed A. Hossayni$^{\mbox{\footnotesize
}}$\thanks{E-mail address: \texttt{hossayni@iran.ir},
(Sayyed A. Hossayni)}\hspace{1mm}}
, \normalsize{J.A. Rad$^{\mbox{\footnotesize
}}$\thanks{E-mail address: \texttt{j.amanirad@gmail.com; j\_amanirad@sbu.ac.ir},
(J.A. Rad)}\hspace{1mm}}
\vspace{-2mm}\\\\
\footnotesize{\em $^{\mbox{\footnotesize a}}$Department$\!$ of$\!$
Computer Sciences, Faculty$\!$ of$\!$ Mathematical Sciences$\!$,}\vspace{-2mm}\\
\footnotesize{\em Shahid Beheshti University, Evin,Tehran 19839,Iran}\\
} \maketitle
\vspace{-.9cm}\noindent\linethickness{.5mm}{\line(1,0){475}}
\maketitle
\begin{abstract}
The objective of this paper is to apply the well-known exact operational matrices (EOMs) idea for solving the Emden-Fowler equations, illustrating the superiority of EOMs versus ordinary operational matrices (OOMs).
Up to now, a few studies have been conducted on EOMs and the differential equations solved by them do not have high-degree nonlinearity and the reported results are not regarded as appropriate criteria for the excellence of the new method.
So, we chose Emden-Fowler type differential equations and solved them by this method.
To confirm the accuracy of the new method and to show the preeminence of EOMs versus OOMs, the norm1 of the residual and error function of both methods are evaluated for multiple $m$ values, where $m$ is the degree of the Bernstein polynomials.
We reported the results in form of plots to illustrate the error convergence of both methods to zero and also to show the primacy of the new method versus OOMs.
The obtained results have demonstrated the increased accuracy of the new method.
\vspace{.3in}
\end{abstract}
\section{Introduction}
In 1915, Galerkin \cite{Galerkin} introduced a broad generalization of the Ritz \cite{Ritz} method to be used primarily for the approximate solution of variational and boundary value problems, including problems that cannot be reduced to variational problems \cite{soviet.enc}. His method was highly appreciated and thousands of problems have been solved, using the Galerkin method, since 1915.\\
The basic idea behind the Galerkin method is as follows. Suppose it is required to find a solution, defined in some domain $U$, of the following (nonlinear) differential equation (which does not have any exact analytical solution)
\begin{align}\label{supN0}
&\mathcal{N}\left[u(x)\right]= 0,&&x\in U\\\nonumber
&B\left(u(s)\right)= 0,&&s\in \partial U
\end{align}
where the solution satisfies at the boundary $B\left(u(s)\right)$ of $U$ the homogeneous boundary conditions.
Firstly, suppose arbitrary so-called "trial" functions $\beta_i(x)$
\begin{align}\label{takhMWR}
&u(x) \approx y_m(x)= \sum_{i=0}^\infty{c_i\beta_i(x)}&B\left(y_m(s)\right)= 0,&&s\in \partial U
\end{align}
As we mentioned, $\mathcal{N}\left[u(x)\right]$ does not have any exact analytical solution; so, we are sure that $y(x)$ cannot satisfy the equation (\ref{supN0}).
Therefore, we attempt to minimize the residual function
\begin{align}\label{jaygMWR}
R_m(x)= \mathcal{N}\left[y_m(x)\right] \neq 0
\end{align}
The main idea of the Galerkin method to minimize the residual function is to find the coefficients $c_i$ so that the residual function becomes zero in average mode
\begin{align}\label{MWREQ}
&(R_m, \beta_j(x))_\omega= \int_U{R_m\beta_j(x)\omega(x)dx}= 0,&&x\in U
\end{align}
in which, $\omega(x)$ is a weight function.\\
Meanwhile, an old and sufficient technique to simplify the computations in the residual function is using the operational matrices.
In 1975, Chen and Hsiao \cite{chen.hsiao.om} introduced an operational matrix to perform integration of Walsh functions.
Chen \cite{walsh.fom} continued his work to introduce some operational matrices to do fractional calculus operations. In 1977, Sannuti et al. introduced the Block-Pulse functions integration operational matrix.
As Mouroutsos \cite{mourout} writes, these studies continued at that time with the determination  of integration operational matrix for miscellaneous basis functions like the Laguerre, the Legendre, the Chebyshev and the Fourier trial functions.
In 1988-1989, Razzaghi et al. \cite{Razzaghi1,Razzaghi.2,Razzaghi.3} presented the integral and product operational matrix based on Fourier series, Taylor series and shifted-Jacobi series.
In 1993, Bolek \cite{bolek.93} has presented a direct method for deriving an operational matrix of differentiation for Legendre polynomials.
In 2000-2012, Yousefi et al. \cite{Razzaghi.Yousefi.1,Razzaghi.Yousefi.2,Yousefi.3,Yousefi.Behroozifar.Dehghan.APM,Yousefi.Behroozifar.Dehghan.CAM,S.A.Yousefi.Behroozifar.2010}, have presented Legendre wavelets and Bernstein operational matrix for solving the variational problems and differential equations.
In 2012, Kazem et al. \cite{Kazem.rad.zna},  has presented a general formulation for the d-dimensional orthogonal functions and their derivative and product matrices has been presented.
In 2013, The authors of \cite{Kazem.rad.MCM} presented a modified form of the homotopy analysis method based on Chebyshev operational matrices.
Kazem \cite{kazem.jacobi.2013} has derived the Jacobi integral operational matrix for the fractional calculus and lots of other new works.\\
Applying the Galerkin method to solve a differential equation, by low-cost computations, we aim to implement the (\ref{jaygMWR}) residual functions computations using the operational matrices.
But, the idea of the derivation of operational matrices does not guarantee the exactitude of the operations done by the matrices.
For a short description, suppose the base functions set and the known functions $f(x)$ and $g(x)$
\begin{eqnarray}\nonumber
\Theta_m= \{\beta_1(x), \beta_2(x), \cdots, \beta_m(x)\}\\\nonumber
f(x)= \sum_{i=1}^m{\kappa_i \beta_i(x)= k^Tb_m(x)},\\\nonumber
g(x)= \sum_{i=1}^m{\lambda_i\beta_i(x)}= l^Tb_m(x),
\end{eqnarray}
where
\begin{eqnarray}\nonumber
b_m(x)= \left[\begin{matrix}\beta_1(x)& \beta_2(x)& \cdots& \beta_m(x)\cr\end{matrix}\right]^T,\\\nonumber
k= \left[\begin{matrix}\kappa_1(x)& \kappa_2(x)& \cdots& \kappa_m(x)\cr\end{matrix}\right]^T,\\\nonumber
l= \left[\begin{matrix}\lambda_1(x)& \lambda_2(x)& \cdots& \lambda_m(x)\cr\end{matrix}\right]^T.
\end{eqnarray}
Ordinary operational matrices (OOMs) employ the differentiation, integration and product of the base vector(s) as
\begin{eqnarray}\nonumber
&&\int_0^x {f(x) dx}= k^T\int_0^x{b_m(x)} \overset{.}{=} k^T P_1 b_m(x)\\\nonumber
&&\frac{d}{dx}f(x)= k^T\frac{d}{dx}{b_m(x)} \overset{.}{=} k^T D_1 b_m(x)\\\nonumber
&&f(x) g(x)= (k^T b_m(x))(b_m(x)^T l) \overset{.}{=} k^T\widehat{L_1} b_m(x)
\end{eqnarray}
where $P_1$, $D_1$ and $\widehat{L_1}$ are integration, differentiation and product matrices related to the base $\Theta_m$ and the symbol ($\overset{.}{=}$) is employed for definition.
As it can be seen, both of the $f(x)$ and the $g(x)$ are in the $Span{(\Theta_m)}$;
but, it is quite probable that their integration or product do not remain in that space;
which means that operational matrices do not guarantee exactitude of the done operations.\\
Recently, Parand et al. \cite{hossayn.tjmcs}, have presented a solution for this problem.
They presented exact operational matrices (EOMs) $P_2$, $D_2$ and $\widehat{L_2}$ so that
\begin{eqnarray}\nonumber
&&\int_0^x {f(x) dx}= k^T\int_0^x{b_m(x)} = k^T P_2 b_{n_1}(x)\\\nonumber
&&\frac{d}{dx}f(x)= k^T\frac{d}{dx}{b_m(x)} = k^T D_2 b_{n_2}(x)\\\nonumber
&&f(x) g(x)= (k^T b_m(x))(b_m(x)^T l)= k^T\widehat{L_2} b_{n_3}(x).
\end{eqnarray}
As it can be seen, all of the approximations have been removed and also the base vector has been changed. $b_{n_i}(x)$ depends on the $b_m(x)$ and the respective operational matrix.\\
Parand et al. \cite{hossayn.tjmcs} have implemented their idea to introduce the Bernstein polynomials operational matrices and solved some simple differential equations by them. The solved differential equations do not show the power of the introduced method well.
The potency of the new method becomes clear, only if, it solves some nonlinear problems with high-degree nonlinearity.
Therefore, in this paper, we chose the Emden-Fowler type differential equations to be solved by the Bernstein polynomials EOMs; so, the EOMs performance becomes apparent.\\
In the section \ref{S_EFEQ}, we briefly introduce the Emden-Fowler equations.
The section \ref{S_B_polynomials} aim to familiarize the reader to the EOMs obtained by \cite{hossayn.tjmcs} like the differentiation matrix $D$, the integration matrix $P$, the product matrix $\widehat{C}$, and the Galerkin matrix $Q$.
Also, we introduce a new "series operational matrix" to achieve the best approximation of $g\left(y(x)\right)$ by the Bernstein polynomials, where $g(x)$ is a determinate function and $y(x)$ is the differential equation unknown function.
At the end of the section, a summary of the solution error analysis proposed in \cite{hossayn.tjmcs} is mentioned, where the problem solution is approximated by the Bernstein polynomials.
In the section \ref{S_application}, 7 Emden-Fowler type equations are solved by the EOM approach.
Also, the results are compared with OOM approach results to prove the validity and applicability of the method and to show the superiority of EOMs in comparison with OOMs.
Finally, in the section \ref{S_conclusion}, a conclusion is brought alongside some new suggestions for further studies.

\section{Emden-Fowler equations}\label{S_EFEQ}
\subsection{Introduction to the equations}
Many problems in the mathematical physics and astrophysics which occur on semi-infinite interval are related to the diffusion of heat perpendicular to the parallel planes and can be modeled by the heat equation:
\begin{eqnarray}\nonumber
x^{-k}\frac{d}{dx}\left(x^k\frac{d}{dx}\right)+ f(x)g\left(y(x)\right)= h(x)~,~~~~~~~x>0,~k>0~,
\end{eqnarray}
or equivalently
\begin{eqnarray}\nonumber
y''(x)+ \frac{k}{x}y'(x)+ f(x)g\left(y(x)\right)= h(x)~,~~~~~~~~~~x>0,~k>0~,
\end{eqnarray}
where $y(x)$ is the temperature. For the steady-state case and for $k= 2$ and $h(x)= 0$, this equation is the generalized Emden-Fowler equation \cite{S.Chandrasekhar-book,H.T.Davis-book,O.U.Richardson-book} given by
\begin{eqnarray}\label{EmFw}
y''(x)+ \frac{2}{x}y'(x)+ f(x)g\left(y(x)\right)= 0~,~~~~~~~x>0~,
\end{eqnarray}
subject to the conditions
\begin{eqnarray}\nonumber
y(0)= a~,~~~~~~~~~~y'(0)= b~,
\end{eqnarray}
where $f(x)$ and $g\left(y(x)\right)$ are two given functions.\\
When $f(x)= 1$, Eq. (\ref{EmFw}) reduces to the Lane-Emden equation which, with specified $g\left(y(x)\right)$, was used to model several phenomena in mathematical physics and astrophysics, such as the theory of stellar structure, the thermal behavior of a spherical cloud of gas, isothermal gas sphere and theory of thermionic currents
\begin{eqnarray}\label{LEm}
&&y''(x)+ \frac{2}{x}y'(x)+ g\left(y(x)\right)= 0~,~~~~~~~~x>0~,\\\nonumber
&&y(0)= a,\text{~~}y'(0)= b~.\nonumber
\end{eqnarray}
For $g\left(y(x)\right)= y^p(x)$, $a= 1$ and $b= 0$ equation (\ref{LEm}) yields the standard Lane-Emden equation that was used to model the thermal behavior of a spherical cloud of gas acting under the mutual attraction of its molecules and subject to the classical laws of thermodynamics \cite{S.Chandrasekhar-book,A.Aslanov,R.P.AgarwalD.ORegan}.
\begin{eqnarray}\label{sLEm}
&&y''(x)+ \frac{2}{x}y'(x)+ y^p(x)= 0~,~~~~~~~~~~~x>0~,\\\nonumber
&&y(0)= 1~,\text{~~}y'(0)= 0~,\nonumber
\end{eqnarray}
where $p>0$ is constant. Substituting $p= 0, 1$ and 5 into Eq. (\ref{sLEm}) leads to the following exact solutions respectively
\begin{eqnarray}\label{exactEmFw1}
y(x)= 1- \frac{1}{3!}x^2~,~~~~~~~y(x)= \frac{\sin{(x)}}{x}~,~~~~~~~y(x)= \left(1+ \frac{x^2}{3}\right)^{-\frac{1}{2}}~.
\end{eqnarray}
In other cases, there is not any analytical exact solution for the standard Lane-Emden equation. Also, for two Emden-Fowler equations (\ref{EteLE}) and (\ref{EeEmFw}), we have the below exact analytical solutions
\begin{align}\label{exactEmFw2}
&y(x)= -2\ln{(1+x^2)}\\
&y(x)= e^{x^2}
\end{align}
In the following subsection, we attempt to going to solve simultaneous Emden-Fowler equations whose $f(x)$ and $g\left(y(x)\right)$ functions are given as
\begin{eqnarray}
&&f(x)= 1~,~~~~~~~~g\left(y(x)\right)= y^p(x)~,~~~~~~p\in\mathbb{N}~,~~~~~~~~~~\left[\begin{matrix}a,~b\cr\end{matrix} \right]= \left[\begin{matrix} 1,~0\cr\end{matrix}\right]\label{EsLE}\\
&&f(x)=1~,~~~~~~~~g\left(y(x)\right)=y^p(x)~,~~~~~~p\notin\mathbb{N}~,~~~~~~~~~~\left[\begin{matrix}a,~b\cr\end{matrix} \right]= \left[\begin{matrix} 1,~0\cr\end{matrix}\right]\label{ErsLE}\\
&&f(x)= 1~,~~~~~~~~g\left(y(x)\right)= e^{y(x)}~,~~~~~~~~~~~~~~~~~~~~~~~~~~~ \left[\begin{matrix}a,~b\cr\end{matrix}\right]= \left[\begin{matrix}0,~0\cr\end{matrix}\right]\label{EseLE}\\
&&f(x)= 1~,~~~~~~~~g\left(y(x)\right)= \sinh{\left(y(x)\right)}~,~~~~~~~~~~~~~~~~~~~\left[\begin{matrix}a,~b\cr\end{matrix} \right]= \left[\begin{matrix}1,~0\cr\end{matrix}\right]\label{EshLE}\\
&&f(x)= 1~,~~~~~~~~g\left(y(x)\right)= \sin{\left(y(x)\right)}~,~~~~~~~~~~~~~~~~~~~~~\left[\begin{matrix}a,~b\cr\end{matrix} \right]= \left[\begin{matrix}1,~0\cr\end{matrix}\right]\label{EsinLE}\\
&&f(x)=1~,~~~~~~~~g\left(y(x)\right)=4\left(2e^{y(x)}+e^{\frac{y(x)}{2}}\right)~,~~~~~~~~~~\left[ \begin{matrix} a,~b\cr\end{matrix}\right]= \left[\begin{matrix}0,~0\cr\end{matrix}\right]\label{EteLE}\\
&&f(x)= -2\left(2x^2+ 3\right)~,~~~~~g\left(y(x)\right)= y(x)~,~~~~~~~~~~~~~~~\left[\begin{matrix} a,~b\cr\end{matrix}\right]= \left[\begin{matrix}1,~0\cr\end{matrix}\right]\label{EeEmFw}
\end{eqnarray}
Several authors have investigated this equation.
In 1989, Bender et al. \cite{C.M.Bender.Milton.Pinsky.Simmons} proposed a perturbative technique for solving nonlinear differential equations such as Lane�Emden.
It consists of replacing nonlinear terms in the Lagrangian such as $y^n$ by $y^{1+\delta}$ and then treating $\delta$ as a small parameter.
In 1993, Shawagfeh \cite{N.T.Shawagfeh} applied a nonperturbative approximate analytical solution for the Lane�Emden equation, using the Adomian decomposition method.
His solution was in the form of a power series. He used Pad\'{e} approximation method to accelerate the convergence of the power series.
In 2001, Mandelzweig and Tabakin \cite{V.B.Mandelzweig.Tabakin} used the quasilinearization approach to solve Lane-Emden equation.
That method approximates the solution of a nonlinear differential equation by treating the nonlinear terms as a perturbation about the linear ones; but, unlike the perturbation theories, it is not based on the existence of some kind of small parameters.
Wazwaz \cite{A.M.Wazwaz.AMC.310} employed the Adomian decomposition method with an alternate framework designed to overcome the difficulty of the singular point.
It was applied to the differential equations of Lane-Emden type.
Thereafter, he used the \cite{wazwaz.173.176} modified decomposition method for solving analytic treatment of nonlinear differential equations such as Lane-Emden equation.
The modified method accelerates the rapid convergence of the series solution, highly, reduces the size of the
work and provides the solution by using few iterations only without any need to the so-called Adomian polynomials.
In 2003, Liao \cite{S.Liao.142.16} provided an analytic algorithm for Lane-Emden type equations. This algorithm logically contains the well-known Adomian decomposition method. Different from all other analytical techniques, this algorithm, itself, provides the researcher with a convenient way to adjust convergence regions even without Pad\'{e} technique.
Applying the semi-inverse method, J. He \cite{J.H.He.143.541} obtained a variational principle for the Lane-Emden equation, which gives much numerical convenience when applying finite element methods or the Ritz method.
In 2004, Parand and Razzaghi \cite{K.Parand.Razzaghi.69.353} presented a numerical technique based on a rational Legendre Tau method to solve higher ordinary
differential equations such as Lane-Emden. In their work, the operational matrices of the derivative and product of rational Legendre functions together with the Tau method were utilized to reduce the solution of these physical problems to the solution of systems of algebraic equations.
Ramos \cite{J.I.Ramos.1,J.I.Ramos.2,J.I.Ramos.3} solved the Lane-Emden equation through different methods. He presented linearization methods for singular
initial value problems in second order ordinary differential equations such as Lane-Emden. These methods result in linear constant-coefficients ordinary differential equations which can be integrated analytically;
thus, they yield piecewise analytical solutions and globally smooth solutions \cite{J.I.Ramos.1}.
Thereafter, in \cite{J.I.Ramos.3}, he presented a series solutions of the Lane-Emden equation based on either a Volterra integral equation formulation or the expansion of the dependent variable in the original ordinary differential equation and compared them with series solutions obtained by means of integral or differential equations based on a transformation of the dependent variables.
Also, in \cite{J.I.Ramos.2}, he developed piecewise-adaptive decomposition methods for the solution of nonlinear ordinary differential equations.
Piecewise-decomposition methods provide series solutions in intervals, which are subject to continuity conditions at the end points of each interval and their adaption is based on the use of either a fixed number of approximants and a variable step size, a variable number of approximants and a fixed step size or a variable
number of approximants and a variable step size.
In 2006, Yousefi \cite{S.A.Yousefi.amc.1417} used integral operator and converted Lane-Emden equations to integral equations and then applied Legendre
wavelet approximations. In his work, properties of the Legendre wavelet together with the Gaussian integration method were utilized to reduce the integral equations to the solution of algebraic equations. By his method, the equation was formulated
on $[0, 1]$.
In 2008, Aslanov \cite{A.Aslanov.372.3555} constructed a recurrence relation for the components of the approximate solution and investigated the convergence
conditions for the Emden-Fowler type equations. He improved the previous results on the convergence radius of the series solution.
Dehghan and Shakeri \cite{Dehghan.Shakeri.13.53} investigated Lane-Emden equation, using the variational iteration method, and showed the efficiency
and the applicability of their procedure for solving this equation. Their technique does not require any discretization, linearization or small perturbations and therefore reduces the volume of computations.
In 2009, Chowdhury and Hashim \cite{M.S.H.ChowdhuryI.Hashim.nar} obtained the analytical solutions of the generalized Emden-Fowler type equations in the second
order ordinary differential equations by homotopy-perturbation method (HPM). This method is a coupling of the perturbation method and the homotopy method. The main feature of the HPM is that it deforms a difficult problem into a set of problems which are easier to solve. HPM yields solutions in convergent series forms with easily computable terms.
In 2009, Bataineh et al. \cite{A.S.Bataineh.S.M.Noorani.Hashim.13.53} obtained analytical solutions of singular initial value problems (IVPs) of the Emden-Fowler type by the homotopy analysis method (HAM). Their solutions contained an auxiliary parameter which provided a convenient
way of controlling the convergence region of the series solutions. It was shown that the solutions, obtained by the Adomian decomposition method (ADM) and the homotopy-perturbation method (HPM), are only special cases of the HAM solutions.
Parand et al. \cite{K.Parand.shahini.Dehghan.JCP} proposed a pseudospectral technique to solve the
Lane-Emden type equations on a semi-infinite domain. The method is based on the rational
Legendre functions and Gauss-Radau integration. Their method reduces the solution of the nonlinear ordinary differential equation to the solution of a nonlinear algebraic equations system.
Also, in recent years, lots of researches have been conducted on the Emden Fowler equations;
the interested reader is referred to the author's recent papers \cite{Parand.Dehghan.Rezaei.Ghaderi,Parand.Abbasbandy.Kazem.Rezaei,Parand-Taghavi,Parand-Pirkhedri,new.astron.Khalique.Ntsime,new.astron.VanGorder,new.astron.VanGorder.2,new.astron.PandeyNarayanKumar,new.astron.BoutanGorder}.
\section{Bernstein polynomials (B-polynomials)}\label{S_B_polynomials}
\subsection{Overview of B-polynomials}
The Bernstein polynomials (B-polynomials) \cite{S.A.Yousefi.Behroozifar.2010}, are some useful polynomials defined on $[0,1]$. The Bernstein polynomials of degree $m$ form a basis for the power polynomials of degree $m$ \cite{mathworld.wolfram}. We can mention lots of their properties. They are continuous over the domain. They satisfy the symmetry
\begin{eqnarray}\nonumber
B_{i,m}(x)= B_{m-i,m}(1-x)
\end{eqnarray}
positivity
\begin{eqnarray}\nonumber
\forall x\in [0,1]:&B_{i,m}(x)\geq 0
\end{eqnarray}
normalization or unity of partition \cite{mathworld.wolfram}
\begin{eqnarray}\nonumber
\sum_{i=0}^m{B_{i,m}(x)}= 1.
\end{eqnarray}
Also, $B_{i,m}(x)$ in which $i\notin \{0,m\}$ has a single unique local maximum of
\begin{eqnarray}\nonumber
i^im^{-m}(m-i)^{m-i}\binom{m}{i}
\end{eqnarray}
occurring at $t=\frac{i}{m}$. All of the B-polynomial bases take 0 value at $x=0$ and $x=1$, except the first polynomial at $x=0$ and the last one at $x = 1$, which are equal to 1. It can provide the flexibility applicable to impose boundary conditions at the end points of the interval.\\
We present the solutions to linear and nonlinear differential equations as
linear combinations of these polynomials $P(x)=\sum_{i=0}^{m} c_{i}B_{i,m}(x)$ and the coefficients $c_{i}$ are determined using the Galerkin method.
In recent years, the B-polynomials have attracted the attention of many researchers.
These polynomials have been utilized for solving different equations by using various approximate methods.
B-polynomials have been used for solving Fredholm integral equations \cite{Chakrabarti.Martha,Mandal.Bhattacharya}, Volterra integral equations \cite{S.Bhattachary.B.Mandal}, Volterra-Fredholm-Hammerstein integral equations \cite{Maleknejad.Hashemizadeh.Basirat.2011},differential equations \cite{S.A.Yousefi.Behroozifar.2010,Singh.Singh.Singh.2009,D.D.Bhatta.Bhatti.2006,M.I.Bhatti.P.Bracken.2007}, integro-differential equations \cite{S.Bhattacharya.Mandal.2008}, parabolic equation subject to specification of mass \cite{S.A.Yousefi.Behroozifar.Dehghan.2010} and so on.
Singh et al. \cite{Singh.Singh.Singh.2009} and Yousefi et al. \cite{S.A.Yousefi.Behroozifar.2010} have proposed operational matrices in different ways for solving differential equations.
In \cite{Singh.Singh.Singh.2009}, the B-polynomials have been firstly orthonormalized using Gram-Schmidt orthonormalization process and then the operational matrix of integration has been obtained.
By the expansion of B-polynomials in terms of Taylor basis, Yousefi and Behroozifar have
found the operational matrices of  differentiation, integration and product of B-polynomials.
\subsection{EOM-related matrices}
\subsubsection{B-polynomials}As we mentioned, $m$-degree B-polynomials\cite{M.I.Bhatti.P.Bracken.2007} are a set of polynomials defined on $[0,1]$:
\begin{eqnarray}\nonumber
B_{i,m}(x)= \binom{m}{i}x^i(1-x)^{m-i}~,~~~~~~~~0\leq i\leq m~,
\end{eqnarray}
where $\binom{m}{i}$ means
\begin{eqnarray}\nonumber
\frac{m!}{i!(m-i)!}~,
\end{eqnarray}
In this paper, we use the $\psi_m(x)$ notation to show
\begin{eqnarray}\nonumber
\psi_m(x)=\left[\begin{matrix}B_{0,m}(x)&B_{1,m}(x)&\cdots&B_{m,m}(x)\cr\end{matrix}\right]^T~,
\end{eqnarray}
We should remember that\cite{S.A.Yousefi.Behroozifar.2010}:
\begin{eqnarray}\label{matrixA}
\psi_m(x)= A_m\times T_m(x),
\end{eqnarray}
where
\begin{eqnarray}\label{TmX}
T_m(x)=\left[\begin{matrix}x^0&x^1&\cdots &x^m\cr\end{matrix}\right]^T~,
\end{eqnarray}
and the $(i+1)^{\text{th}}$ row of matrix $A$ is
\begin{eqnarray}\label{adiffenetion}
A_{i+1}= \begin{bmatrix}
\overbrace{\begin{matrix}
0& 0& \cdots& 0
\end{matrix}}^{\text{i times}}& (-1)^0\binom{m}{i}&(-1)^1\binom{m}{i}\binom{m-i}{1}&\cdots&(-1)^{m-i}\binom{m}{i}\binom{m-i}{m-i}~\cr
\end{bmatrix}~.
\end{eqnarray}
Matrix $A$ is an upper triangular matrix and $\mathrm{det}(A)=\prod_{i=0}^{m}\binom{m}{i}$
; So, $A$ is an invertible matrix. \cite{hossayn.tjmcs} has obtained a relation for obtaining the $A^{-1}$.
\begin{eqnarray}\label{ainverse}
\{\mathbf{A}^{-1}\}_{i,j=0}^{m}
=
\begin{cases}
\frac{{m-i \choose j-i}}{{m \choose j}}~, & j \geq i~,\\
0~, & j < i~.
\end{cases}
\end{eqnarray}
Now, that is the turn to familiarize the reader with the exact operational matrices, introduced in \cite{hossayn.tjmcs}.
\subsubsection{A general formula for $x^i$}
The term $x^i$ is a very common term in differential equations. So, the authors of \cite{hossayn.tjmcs} have proposed a general formula for $x^i$ to be written as linear combination of the Bernstein polynomials
\begin{eqnarray}
&x^i= d_{i,m}^T\psi_m(x)~,~~~~~~m\geq i~, \label{dM}
\end{eqnarray}
\begin{eqnarray*}
&d_{i,m}=\left(\begin{bmatrix}
\overbrace{\begin{matrix}
0&0&\cdots&0\cr
\end{matrix}}^{i}&1&\overbrace{\begin{matrix}
0&0&\cdots&0\cr
\end{matrix}}^{m-i}
\end{bmatrix}A_m^{-1}\right)^T~.
\end{eqnarray*}
\subsubsection{K matrices}
Firstly, \cite{hossayn.tjmcs} has introduced two simple matrices and named them the K-matrices:
\begin{eqnarray}\label{Kmi}
K_{m,i}=\left[\begin{matrix}I_m&0_{m\times i}\cr\end{matrix}\right]_{m\times m+i}~,\\
K^\prime_{m,i}=\left[\begin{matrix}0_{m\times i}&I_m\cr\end{matrix}\right]_{m\times m+i}~.\label{Kpmi}
\end{eqnarray}
\subsubsection{The increaser matrix}\label{S_increaserMatrix}
Suppose that we want to solve the differential equation (\ref{supN0}), using the Galerkin method. To implement (\ref{jaygMWR}) by EOMs, we apply EOMs into the equation and sum all of the terms together to reach the residual function.
To be able to factor out a base vector from all of the different $b_n(x)$-sized terms in the residual function, \cite{hossayn.tjmcs} has introduced the so-called increaser matrix $E_{i,j}$, by which:
\begin{eqnarray}\nonumber
b_i(x)= E_{i,j}.b_j(x).
\end{eqnarray}
Using this matrix, we can convert the base vector existing in each term to the biggest-sized $b_{maxNum}(x)$. By factoring out the $b_{maxNum}(x)$, we can write the residual function as $R_{maxNum}b_{maxNum}(x)$.
So, to solve the problem, we should solve the following:
\begin{align}\label{baseRes}
&R_{{maxNum}}b_{maxNum}(x)= 0.&&{maxNum}\geq m
\end{align}
\cite{hossayn.tjmcs} has, also, obtained the increaser matrix for the Bernstein base functions
\begin{align}\label{Emi}
&\psi_m(x)= E_{m,i}\psi_{m+i}(x)~.\\\nonumber
&E_{m,i}= A_mK_{m+1,i}A_{m+i}^{-1}~.
\end{align}
The $E_{m,i}$ matrix size is $(m+1)\times (m+1+i)$.
\subsubsection{The differentiation matrix}
$D_m$ is the operational matrix of differentiation for the Bernstein base functions, introduced in \cite{hossayn.tjmcs}
\begin{align}
&\frac{d}{dx}\psi_m(x)= D_m\psi_{m-1}(x)~,\label{Dm}\\\nonumber
&D_m= m~\Bigg({K'}_{m,1}^T- K_{m,1}^T\Bigg)~,
\end{align}
\subsubsection{The integration Matrix}
\cite{hossayn.tjmcs} has introduced the Bernstein polynomials integration operational matrix $P_m$ as follows
\begin{eqnarray}\label{Pmatrix}
\int_0^x{\psi_m(x)dx}= P_m\psi_{m+1}(x)~,
\end{eqnarray}
where
\begin{align*}
&P_m= \begin{bmatrix}
p_{0,m}&\cdots& p_{m,m}
\end{bmatrix}_{(m+1)\times (m+2)}^T~.\\
&p_{i,m}= \frac{1}{m+1}\begin{bmatrix}
\overbrace{\begin{matrix}0&\cdots &0\cr\end{matrix}}^{i+1}\overbrace{\begin{matrix}1&\cdots&1\cr\end{matrix}}^{m+1-i}
\end{bmatrix}^T~,
\end{align*}
\subsubsection{The product matrix}
For an arbitrary vector $c$, we can write:
\begin{eqnarray}\label{cTilde}
c^T\psi_m(x)\psi_n^T(x)= \psi_{m+n}^T(x)\times \tilde{C}_{m,n}~,
\end{eqnarray}
where $\tilde{C}_{m,n}$ is an $(m+n+1)\times (n+1)$ product operational matrix for the vector $c$, introduced in \cite{hossayn.tjmcs}
\begin{eqnarray}
\tilde{C}_{m,n}= \begin{bmatrix}
c_0a_{0,0,m,n}&0&\cdots&\cdots&\cdots& 0\cr
c_1a_{1,0,m,n}&c_0a_{0,1,m,n}&\ddots&\ddots&\ddots&\vdots\cr
\vdots&\ddots&\ddots&\ddots&\ddots&\vdots\cr
c_ja_{j,0,m,n}&c_{j-1}a_{j-1,1,m,n}&\ddots& c_0a_{0,j,m,n}&\ddots&\vdots\cr
\vdots&\ddots&\ddots&\ddots&\ddots&0\cr
c_ma_{m,0,m,n}&c_{m-1}a_{m-1,1,m,n}&\ddots& c_{m-j+1}a_{m-j+1,j,m,n}&\ddots&c_0a_{0,n,m,n}\cr
0& c_ma_{m,1,m,n}&\ddots &c_{m-j+2}a_{m-j+2,j,m,n}&\ddots &c_1a_{1,n,m,n}\cr
\vdots&\ddots&\ddots&\ddots&\ddots&\vdots\cr
\vdots&\ddots&\ddots&c_ma_{m,j,m,n}&\ddots& c_{j-1}a_{j-1,n,m,n}\cr
\vdots&\ddots&\ddots&\ddots&\ddots&\vdots\cr
0&\cdots&\cdots&\cdots&0&c_ma_{m,n,m,n}
\end{bmatrix}
\end{eqnarray}
or
\begin{eqnarray}\nonumber
\left[\tilde{C}_{m,n}\right]_{i,j}= \begin{cases}
0~,&i\not\in[j,j+m]~,\cr
c_{i-j}a_{(i-j),(j-1),m,n}~,&o.w
\end{cases}
\end{eqnarray}
Moreover, by transposing (\ref{cTilde}), we have:
\begin{eqnarray}\label{cHat}
\psi_n(x)\psi_m^T(x) c= \widehat{C_{n, m}}\psi_{m+n}(x)~,
\end{eqnarray}
\begin{eqnarray}\nonumber
\widehat{C_{n,m}}= \tilde{C}_{m,n}^T~,
\end{eqnarray}
$\widehat{C_{n,m}}$ is, also, called the product operational matrix for the vector $c$ \cite{hossayn.tjmcs}.
\subsubsection{The power matrix}
Suppose that $y(x)=c^T\psi_m(x)$; \cite{hossayn.tjmcs} has introduced $\overline{C_{m,p}}$ operational matrix by which $y^{p}(x)=  \overline{C_{m,p}}\phi_{p\cdot m}(x)$ and named it the power operational matrix for the Bernstein polynomials
\begin{align}\label{overL}
&\overline{C_{m,p}}= c^T\prod_{i=1}^{p-1}{\widehat{C_{i\cdot p,m}}}~,&&p\geq 2~.
\end{align}
\subsubsection{The Series matrix}
Suppose that $y(x)= c^T\psi_m(x)$; we are to propose a matrix to approximate $f\left(y(x)\right)$ function, where power series of $f(x)$ can be written as:
\begin{align*}
f(x)&= \sum_{i=0}^\infty{e_i x^{i}}&&|x|< R\\
&\simeq \sum_{i=0}^N{e_i x^{i}}&&|x|< R\\
f\left(y(x)\right)\simeq& \sum_{i=0}^N{e_i\left(c^T\psi_m(x)\right)^{i}}&&\left| y(x)\right| < R&\\
&\overset{\ref{overL}}{=} e_0\cdot d_{0,m\cdot N}^T\psi_{m\cdot N}+ e_1\cdot c^T\psi_m(x)+ \sum_{i=2}^N{e_i\overline{C_{m,i}}\psi_{m\cdot i}(x)}&&\left| y(x)\right| < R,\text{\tab}N\geq 2\\
&\overset{\ref{Emi}}{=} \left(e_0\cdot d_{0,m\cdot N}^T+ e_1\cdot c^TE_{m,m(N-1)}+ \sum_{i=2}^N{e_i\overline{C_{m,i}}E_{m\cdot i, m(N- i) }}\right)\psi_{m\cdot N}(x)&&\left| y(x)\right| < R,\text{\tab}N\geq 2\\
\end{align*}
\begin{align}
\Longrightarrow f\left(y(x)\right)&\simeq \overbrace{C_{e_i, m, N}}\psi_{m\cdot N}(x),\label{MLM}&&\left| y(x)\right| < R&N\geq 2\\
\overbrace{C_{e_i, m, N}}&= e_0\cdot d_{0,m\cdot N}^T+ e_1\cdot c^TE_{m,m(N-1)}+ \sum_{i=2}^N{e_i\overline{C_{m,i}}E_{m\cdot i, m(N- i) }}\nonumber
\end{align}
In the next subsection, the optimal $e_i$ coefficients are obtained.
\subsubsection{Truncated Taylor series}
In the last subsection, we put the base of the series matrix on the power series.
We know that $f(x)= \sum_{i=0}^{\infty}c_ix^i,~c_i=\frac{f^{(i)}}{i!}$; but after truncating the series to $N$ terms, the $c_i$s would not be, necessarily, the best coefficients to approximate $f(x)$.
To reach the best coefficients, we need the following theorem
\begin{thm}
Consider the Hilbert space $H= L^2[a,b]$ and one of its finite-dimensional subspaces:
\begin{align*}
F=Span\{f_0(x), f_1(x), \cdots, f_N(x)\},
\end{align*}
with inner product defined by
\begin{align*}
\langle f(t), g(t)\rangle =\int_a^b{f(t)g(t)dt}.
\end{align*}
\begin{enumerate}
\item
For any arbitrary function $q(x)\in H$, there exists a unique best approximation $p(x)$ (in respect to the defined inner product) for the $q(x)$ function
\item
\begin{align*}
&\langle \left(q(x)- p(x)\right) , f \rangle= 0, \\
&f= \begin{bmatrix}
f_0(x)&f_1(x)&\cdots &f_N(x)\cr
\end{bmatrix}^T
\end{align*}
where
\begin{align*}
\langle f(x), f \rangle= \begin{bmatrix}
\langle f(x), f_0(x) \rangle &\langle f(x), f_1(x) \rangle &\cdots &\langle f(x), f_N(x) \rangle\cr
\end{bmatrix}.
\end{align*}
\end{enumerate}
\end{thm}
\begin{proof}
Refer to \cite{S.A.Yousefi.Behroozifar.2010}
\end{proof}
Consider the Hilber space $H=L^2[a,b]$.
Approximating the function $f(x)\in H$ by a truncated series in the interval $\begin{bmatrix}
a& b\cr
\end{bmatrix}$ means approximating it by a function $g(x)$ in the $T_m$ subspace of $H$:
\begin{align*}
T_m=Span\{1, x, \cdots, x^N\},
\end{align*}
Where
\begin{align*}
&g(x)= e^Tt_N,\\
&e= \begin{bmatrix}
e_0&e_1&\cdots &e_N\cr
\end{bmatrix}^T,\\
&t_N= \begin{bmatrix}
1&x&x^2&\cdots &x^N\cr
\end{bmatrix}^T
\end{align*}
Using the theorem, we have:
\begin{align*}
&\langle \left(f(x)- e^Tt_N\right) , t_N \rangle= 0\\
&e^T\langle t_N, t_N \rangle= \langle f(x), t_N\rangle
\end{align*}
$\langle t_N, t_N\rangle$ is an $(N+1)\times (N+1)$ matrix and is said dual matrix of $t_N$. Let:
\begin{align*}
U= \langle t_N, t_N\rangle= \int_a^bt_Nt_N^Tdx,
\end{align*}
Then,
\begin{align}\label{eFinder}
e^T= \left(\int_a^b f(x)t_N^T(x)dx\right)U^{-1},
\end{align}
So, the best approximation for $f(x)$ would be:
\begin{align*}
f(x)= \sum_{i=0}^N e_ix^i,
\end{align*}
where $e_i$s are the elements of the vector $e$, obtained above.
\subsubsection{The Q Matrix}\label{S_Q}
In the subsection \ref{S_increaserMatrix}, we converted the solution of the differential equation (\ref{supN0}) to the solution of the algebraic equations system (\ref{baseRes}). In (\ref{baseRes}), $b_{maxNum}(x)$ is a base vector and its functions are linearly-independent; so we can solve the following equation instead:
\begin{align}\nonumber
&R_{maxNum}= 0&&{maxNum}\geq m
\end{align}
To overcome the problems of solving a system with ${maxNum}$-equations and $m$ unknown variables, \cite{hossayn.tjmcs} has introduced the Galerkin matrix $Q_{{maxNum},m}$, which reduces the number of equations to $m$, based on the Galerkin method:
\begin{eqnarray}\label{ress}
R_{1\times ({maxNum}+1)}\times Q\left({{maxNum}, m}\right)= 0~,
\end{eqnarray}
where $Q\left({{maxNum}, m}\right)$ is an $({maxNum}+ 1) \times (m+ 1)$ matrix
\begin{align*}
Q\left({{maxNum}, m}\right)&= [q_{ij}]_{{maxNum}+1\times m+1},&&i\leq {maxNum}\\
& &&,~j\leq m~,\\
q_{ij}&= \frac{\binom{{maxNum}}{i-1}\binom{m}{j-1}\left({maxNum}+m+2-(i+j)\right)!+\left(i+j-2\right)!}{\left({maxNum}+m+1\right)!}~.
\end{align*}
then, using Eq. (\ref{ress}), we solve
\begin{eqnarray}\label{residual}
R^*_m= 0~,
\end{eqnarray}
where
\begin{eqnarray}\nonumber
R^*_m= R_{1\times ({maxNum}+1)}\times Q\left({{maxNum}, m}\right)~.
\end{eqnarray}
By solving the obtained algebraic system, we will find the $m+1$ unknown coefficients $c_i$, in \ref{takhMWR}, and, finally, find the $y_m(x)$.
\section{Applications}\label{S_application}
To show the efficiency and accuracy of the present method on ordinary differential equations, seven examples are presented.
The exact solution is not available for most of them;
therefore, we solved them using a seventh-eighth order continuous Runge-Kutta method as an almost exact solution, using the Maple$^\copyright$ dverk78 function, and compared them with the EOM method results.
In order to reach the approximated solutions, we applied the present method for different $m$ valued $y_m(x)$s in (\ref{takhMWR}).
We have, also, compared them with their (almost) exact solutions and computed the norm1 for the residual and the error function of each one.\\
The numerical implementation and all of the executions are performable by maplesoft$^\copyright$.maple.16.x64, 64-bit Microsoft$^\copyright$ Windows7 Ultimate Operating System, alongside hardware configuration: Laptop 64-bit Core i3 M380 CPU, 8 GBs of RAM.
\subsection{Solution of the Emden fowler equations}
Rewriting the Emden-Fowler equation (\ref{EmFw}), we would have
\begin{align}\label{rEmFw}
N\left[y(x)\right]= &xy''(x)+ 2y'(x)+ xf(x)g\left(y(x)\right)= 0,&&0\leq x\leq \lim_{M\rightarrow \infty}M&\\
&y(0)= a\text{~~}y'(0)= b\nonumber
\end{align}
Neglecting the large values of $M$, we can solve the equation in the $0\leq x\leq M$ interval.
So, we solve the problem in the domain $[0,M]$; but the Bernstein polynomials domain is [0,1] and we are not able to solve the problem, yet.
One approach to solve this problem, is to change the variables, using the following mapping
\begin{eqnarray}\nonumber
s=\frac{x}{M}~,~~~~~~~~v(s)= \frac{y(x)}{M}~.
\end{eqnarray}
Then, we have
\begin{align*}
&y(x)= Mv(s)\\
&\frac{d}{dx}y(x)= \frac{d}{ds}v(s)\\
&\frac{d^2}{dx^2}y(x)= \frac{1}{M}\frac{d^2}{ds^2}v(s)\\
&v(0)= \frac{a}{M}\\
&v'(0)= b\\
\end{align*}
Now, rewriting the equation (\ref{rEmFw}), we have
\begin{align*}
N\left[y(x)\right]= &sv''(x)+ 2v'(s)+ sMf\left(sM\right)g\left(Mv(s)\right)= 0,&&0\leq s\leq 1&s= \frac{x}{M}\\
&v(0)= \frac{a}{M}\text{~~}v'(0)= b\nonumber
\end{align*}
First of all, suppose that $z''(s)$ is an approximation for $v''(s)$
\begin{align}\nonumber
z''(s)= c^T\psi_m(s)\Longrightarrow&\begin{cases}
sz''(s)\overset{\ref{dM}}{=} c^T\psi_m(s)\psi_1^T(s)d_{1,1}\overset{\ref{cHat}}{=} c^T\widehat{\left(D_{1,1}\right)_{m,1}}\psi_{m+1}(s)\\
z'(s)- z'(0)= c^TP_m\psi_{m+1}(s)
\end{cases}\\
\Longrightarrow&z'(s)= g^T\psi_{m+1}(s)\nonumber\\
&g\overset{\ref{dM}}{=} P_m^Tc+ b\cdot d_{0,m+1}\nonumber\\
\Longrightarrow&z(s)- z(0)= g^TP_{m+1}\psi_{m+2}(s)\nonumber\\
\Longrightarrow&z(s)= h^T\psi_{m+2}(s)\label{EmFwZs}\\
&h\overset{\ref{dM}}{=} P_{m+1}^Tg+ \frac{a}{M}\cdot d_{0,m+2}\nonumber
\end{align}
Depending on the Emden-Fowler problem which we are to solve, there exist two vectors $k$ and $s(c)$, ($s(c)$ elements depend on the vector $c$ elements) and integers $i$ and $j$ by which we can express $sMf(sM)= k^T\psi_{i}(s)$ and $g\left(Mz(s)\right)= s^T(c)\psi_j(s)$. So we can write
\begin{align}\label{3pEmFw}
&g\left(Mz(s)\right)\cdot sMf(sM)= s^T(c)\psi_j(s)\psi_i^T(s)k\overset{\ref{cHat}}{=} s^T(c)\widehat{K_{j,i}}\psi_{i+j}(s)
\end{align}
Using (\ref{rEmFw}), we can write the Residual$(s)$
\begin{align}
\text{Residual}(s)&= N\left[z(s)\right]\nonumber\\
&= sz''(s)+ 2z'(s)+ sMf(sM)g\left(Mz(s)\right)\nonumber\\
&\overset{\ref{3pEmFw}}{=} c^T\widehat{\left(D_{1,1}\right)_{m,1}}\psi_{m+1}(s)+ 2g^T\psi_{m+1}(s)+ s^T(c)\widehat{K_{j,i}}\psi_{i+j}(s)\nonumber\\
&= R_{1\times{maxNum}}\psi_{{maxNum}}(s)&{maxNum}= \text{max}(m+1, i+j)
\end{align}
\begin{align}
R&\overset{\ref{Emi}}{=} \begin{cases}
c^T\widehat{\left(D_{1,1}\right)_{m,1}}E_{m+1,i+j-(m+1)}+ 2g^TE_{m+1,i+j-(m+1)}+ s^T(c)\widehat{K_{j,i}}&i+j> m+1\\ \\
c^T\widehat{\left(D_{1,1}\right)_{m,1}}+ 2g^T+ s^T(c)\widehat{K_{j,i}}E_{i+j,m+1-(i+j)} &i+j< m+1\\ \\
c^T\widehat{\left(D_{1,1}\right)_{m,1}}+ 2g^T+ s^T(c)\widehat{K_{j,i}}&i+j= m+1
\end{cases}\label{ResEmFw}
\end{align}
Using (\ref{residual}), we solve the following system to find the unknown $c_i$s (elements of the vector $c$)
\begin{align}\label{finalSystem}
R_m^*= R\times Q\left({{maxNum}}, m\right)= 0
\end{align}
As it is obvious from (\ref{ResEmFw}), the only remaining step to gain residual$(s)$ is to determine the vectors $s(c)$ and $k$. Looking at the equations (\ref{EsLE})-(\ref{EeEmFw}), we can write
\begin{align*}
&sMf(sM)= \begin{cases}
M\cdot d_{1,1}^T\psi_1(s)&(\ref{EsLE})-(\ref{EteLE})\\
-2M\left(2M^2d_{3,3}^T+3d_{1,3}^T\right)\psi_{3}(x)&(\ref{EeEmFw})
\end{cases}
\end{align*}
$k$ vectors were found for all of the equations (\ref{EsLE})-(\ref{EeEmFw}).
Now, we have to find the $s(c)$ vectors.
Using (\ref{EmFwZs}), besides (\ref{overL}), we can write
\begin{align*}
&g\left(Mz(s)\right)= \begin{cases}
M^pz^p(s)\overset{\ref{overL}}{=} M^p\overline{H_{m+2, p}}\psi_{(m+2)p}(s)&p\in\mathbb{N}\text{\tab\tab~~~~~~}(\ref{EsLE})\\
Mz(s)= Mh^T\psi_{m+2}(s)&\text{\tab\tab\tab\tab~}(\ref{EeEmFw})
\end{cases}
\end{align*}
for other $g$ functions $(\ref{ErsLE})-(\ref{EteLE})$, Using (\ref{EmFwZs}), besides (\ref{MLM}) and (\ref{eFinder}), we can write
\begin{align*}
g\left(y(x)\right)\simeq \overbrace{H_{e_i, m, N}}\psi_{(m+2)N}(x),
\end{align*}
where $e_i$s are the elements of vector
\begin{align*}
e\overset{\ref{eFinder}}{=} \left(U^{-1}\right)^T\int_a^bg(x)t_N(x)dx.
\end{align*}
Now, we are able to solve the (\ref{finalSystem}) system and gain the vector $c$. After finding the vector $c$, we would have the vector $h$ in (\ref{EmFwZs}); so
\begin{align*}
y(x)&= Mv(s)\\
&\simeq Mz(s)\\
&= Mh^T\psi_{m+2}(s)
\end{align*}
It is worth mentioning that it is too difficult to solve the (\ref{finalSystem}) system of nonlinear equations even by the Newton's method when the number of algebraic equations rises up; the main difficulty with such systems is how to choose the initial guess to handle the Newton method such that it results in low order computations. So, we applied a technique, introduced in \cite{hossayn.tjmcs}, to choose an appropriate initial guess.\\
However, we solved all of the (\ref{EsLE})-(\ref{EeEmFw}) problems, using the Maple$^\copyright$ mathematical software. The results are depicted in the figures (\ref{Fig_sXM1_8_8_s})-(\ref{Fig_sin_5_12_last}).
All of the problems, except the (\ref{EseLE}) are solved for several $m$ values
\begin{align*}
m\in \{2,3,\cdots,E\}
\end{align*}
Where, E is the largest $m$ for which the results are reported.
To evaluate the solution accuracy for each $m$, it is not needed to check them one by one;
reporting the norm of the residual and the error functions is enough for evaluation;
so, we present two categories of figures. The first category contains the figures which report the norm of the residual or the error function, for all of the $m$ values, and the second category contains the plots of the error or the residual function for the largest $m$ value ($E$),\\
Looking at the first category plots, we find out that both of the residual and the error function values decreases by increasing the $m$ value which confirms the convergence of both methods.
Also, they demonstrate the absolute priority of EOM.
the second category of plots shows that almost for all $x$ values, both of the residual and the error functions values of the EOM method are less than OOM.
Meanwhile, the second category plots are semi-logarithmic, because, according to the high superiority of EOMs to OOMs, the error of the EOM method results can not be seen obviously, in non-logarithmic plots.\\
To have a good approximation of $g(y(x))$, we have used some large $N$ values like 8, 10, 12 and 15.
The acceptable results accuracy shows the efficient choices for $N$.\\
Before explaining the figures one by one, we should remind that the problem is solved in the interval $x\in [0, M]$. The question is what value we should choose for $M$?
Like the previous researches, we are only interested in the positive $y(x)$ values (we only solve the problems with positive $y(0)$ values).
Remember that we have the (almost) exact solutions for all of the problems.
Accordingly, we chose the $M$ as the point at which $y(M)=0$.
But if $M>5$, then we set $M=5$.\\
Here, we explain the plots one by one.
The analytical solution of the equation (\ref{EsLE}) for $y^0(x)$ is so that we can exactly write it as a linear combination of the Bernstein polynomials; so, we do not report its figures. But, reports of the problem for $y^1(x)$ and $y^5(x)$ are depicted in the figures \ref{Fig_sXM1_8_8_s},\ref{Fig_sXM1_8_8_last},\ref{Fig_sXM5_8_8_s},\ref{Fig_sXM5_8_8_last}.
The figures \ref{Fig_42exp_5_10_s},\ref{Fig_42exp_5_10_last} show less accuracy than other problems.
This shows that $N$ value is not sufficiently large to approximate $g(y(x))$, well.
However, these kinds of results are appropriate to show the truncated series performance in function approximation.
In the figures \ref{Fig_2LM_8_8_s},\ref{Fig_2LM_8_8_last}, which contains the (\ref{EeEmFw}) problem results, the analytical solution is compared with both methods approximate solutions and, obviously, shows the preference of the EOMs to OOMs, as for the previous problems.
The problem (\ref{ErsLE}) is solved twice, one time for $g(y(x))= y^{\frac{3}{2}}(x)$ and the second time for $g(y(x))= y^{\frac{5}{2}}(x)$ and the results are depicted in the figures \ref{Fig_xM1A5_5_12_s},\ref{Fig_xM1A5_5_12_last},\ref{Fig_xM2A5_5_12_s},\ref{Fig_xM2A5_5_12_last}.
The figures \ref{Fig_xM1A5_5_12_last} and \ref{Fig_xM2A5_5_12_last} shows the ineligible results for the OOM method residual function.
For example, in $x=0$, the error is even more than 1, while, even in those points, the EOM method presents some suitable results.
The interesting point in the solution of the problem (\ref{EseLE}) for $m=2$, which is observable in the figure \ref{Fig_exp_res_2_15_last}, is the ability to reach such high accuracy approximations, even for such a small $m$, by choosing a large-enough value for $N$, which in this problem is 15.
The figures \ref{Fig_sinh_5_10_s},\ref{Fig_sinh_5_10_last},\ref{Fig_sin_5_12_s},\ref{Fig_sin_5_12_last} as solution figures of the problems (\ref{EshLE}) and (\ref{EsinLE}) again, show some almost ineligible results for OOM method, while EOMs have some suitable results.
\begin{figure}
\centering
\subfigure[$\fNorm{Error\left(y_m(x)\right)}_1$]{
\centering
\includegraphics[scale=0.4]{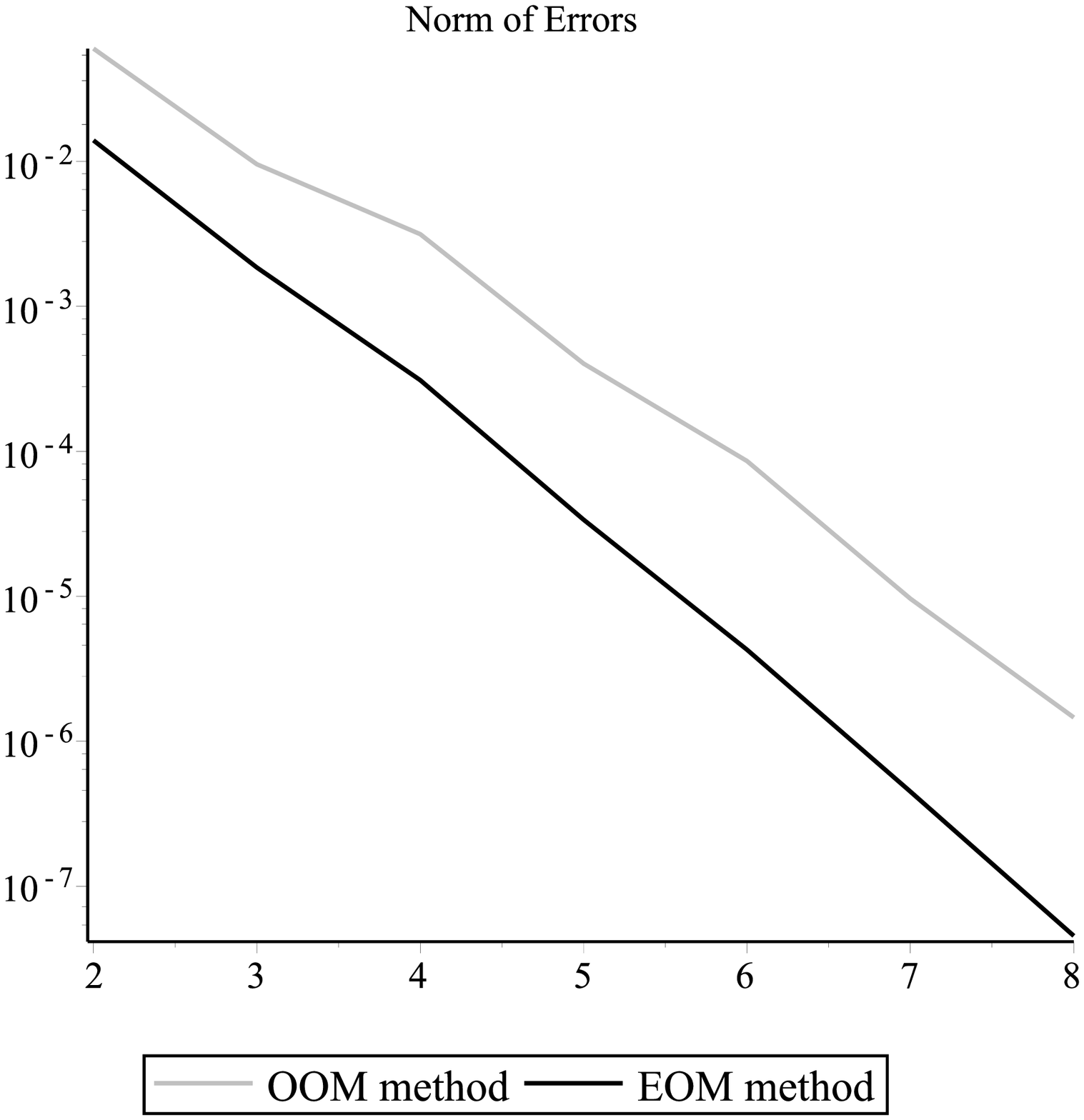}
\label{Fig_sXM1_8_8_errs}
}
\subfigure[$\fNorm{Residual\left(y_m(x)\right)}_1$]{
\centering
\includegraphics[scale=0.4]{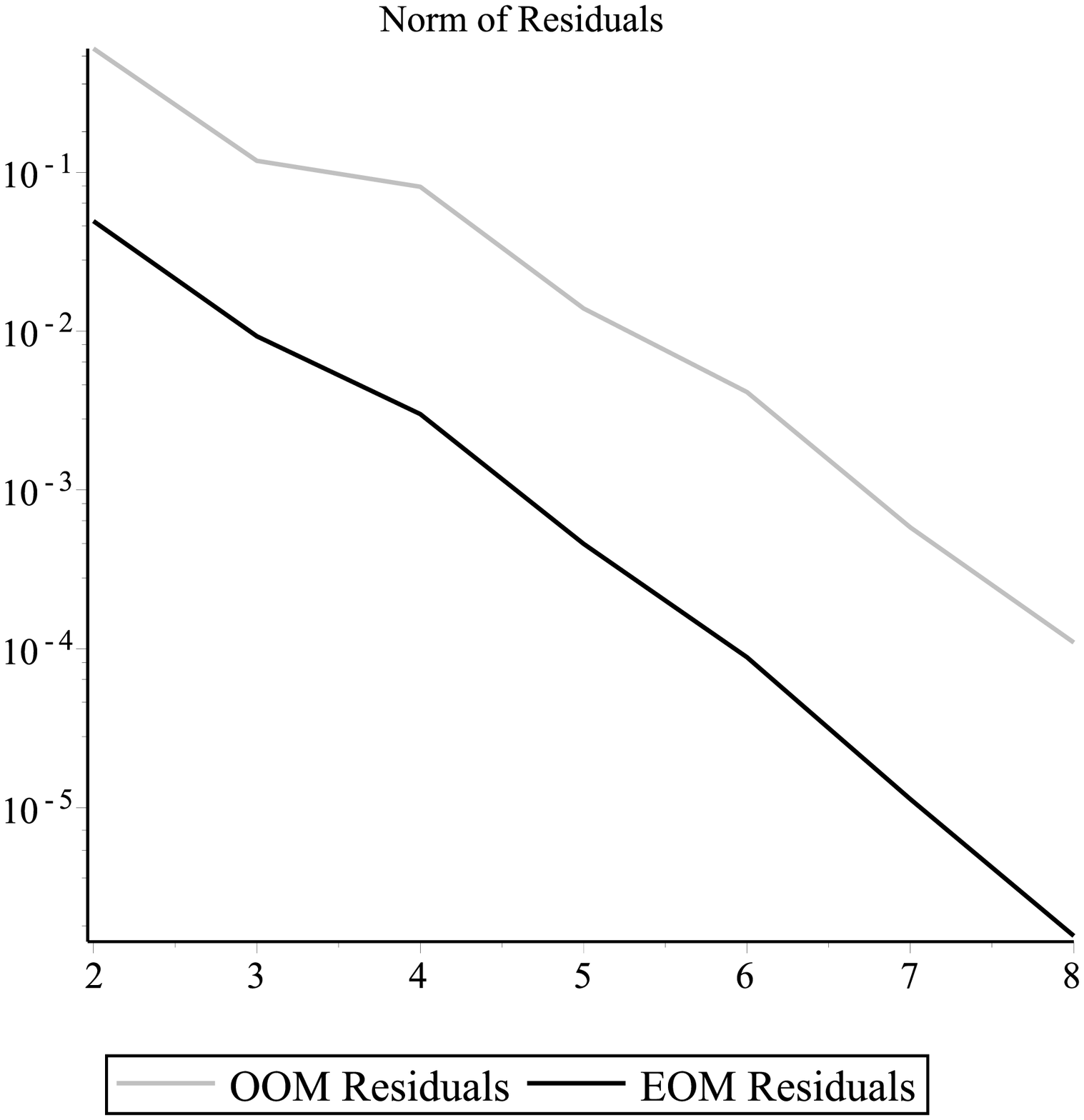}
\label{Fig_sXM1_8_8_ress}
}
\caption{
the norm1 of the residual and error function plots for several $m$ values and $f(x)= 1,~g(x)= y(x),~N=8$}
\label{Fig_sXM1_8_8_s}
\end{figure}
\begin{figure}
\centering
\subfigure[$\fNorm{Error\left(y_M(x)\right)}_1$]{
\centering
\includegraphics[scale=0.4]{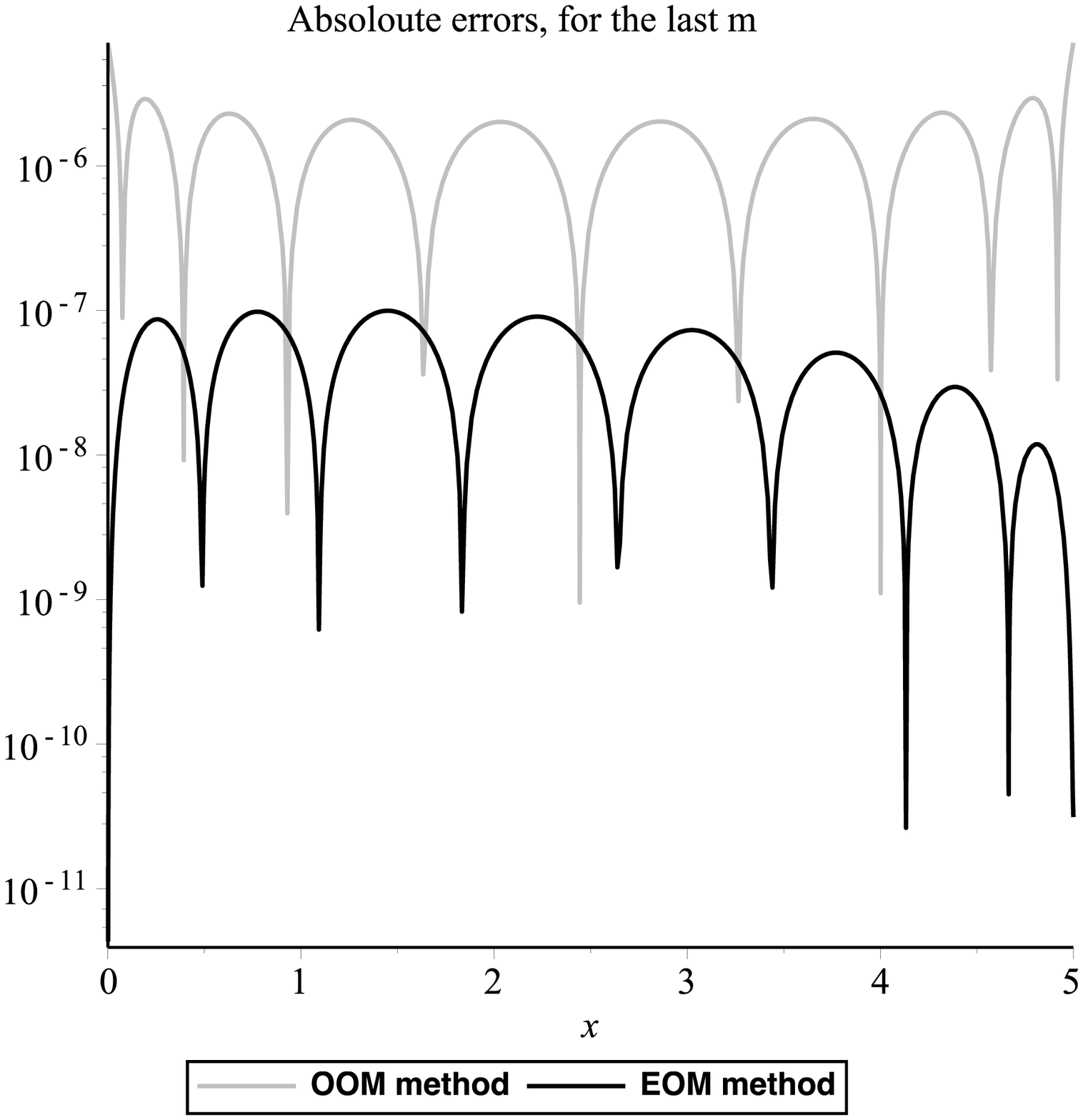}
\label{Fig_sXM1_8_8_last_err}
}
\subfigure[$\fNorm{Residual\left(y_M(x)\right)}_1$]{
\centering
\includegraphics[scale=0.4]{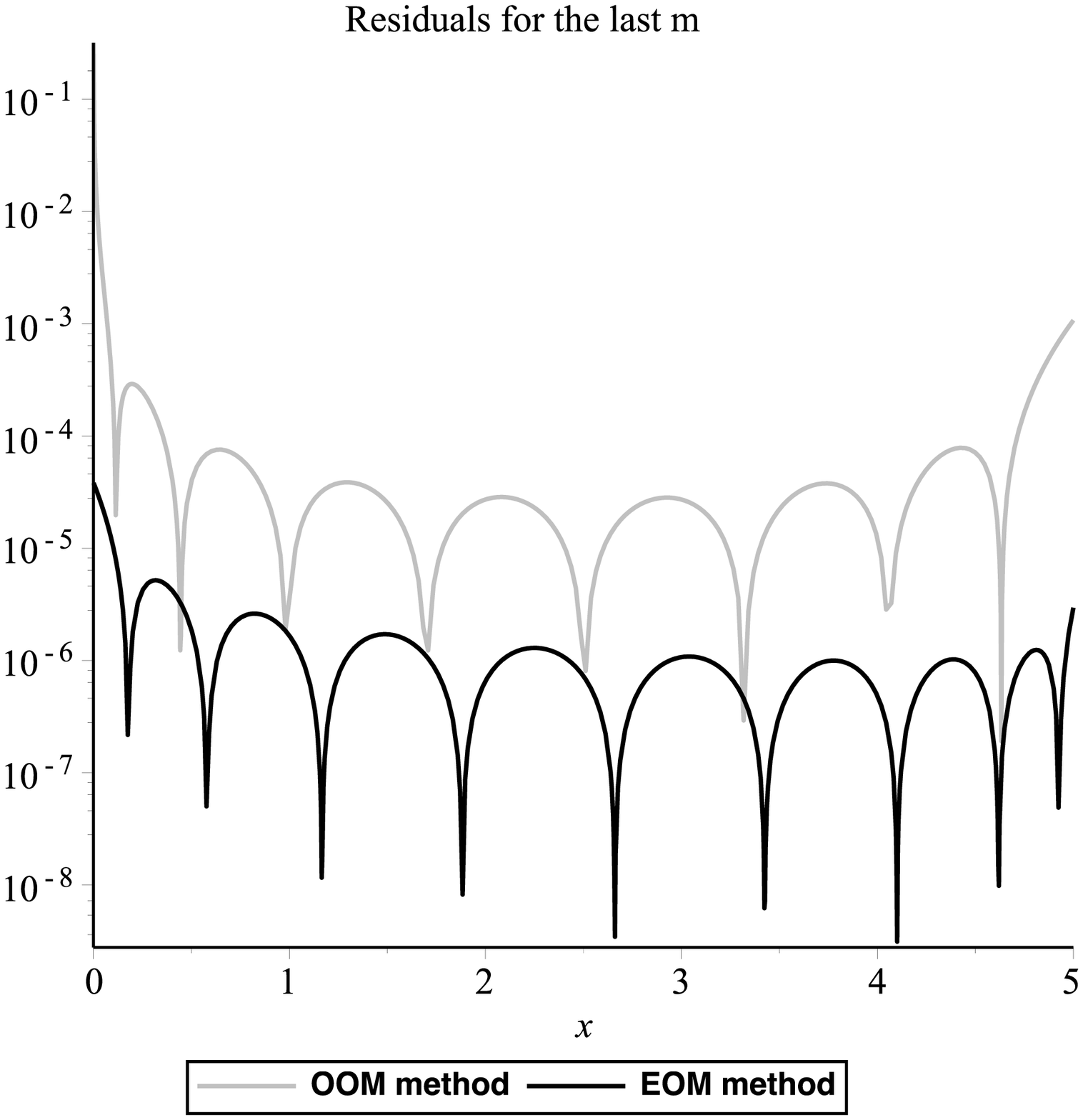}
\label{Fig_sXM1_8_8_last_res}
}
\caption{
the norm1 of the residual and error function plots for the largest $m$ value ($E$) of the figure (\ref{Fig_sXM1_8_8_s}) for $f(x)= 1,~g(x)= y(x),~N=8$}
\label{Fig_sXM1_8_8_last}
\end{figure}
\begin{figure}
\centering
\subfigure[$\fNorm{Error\left(y_m(x)\right)}_1$]{
\centering
\includegraphics[scale=0.4]{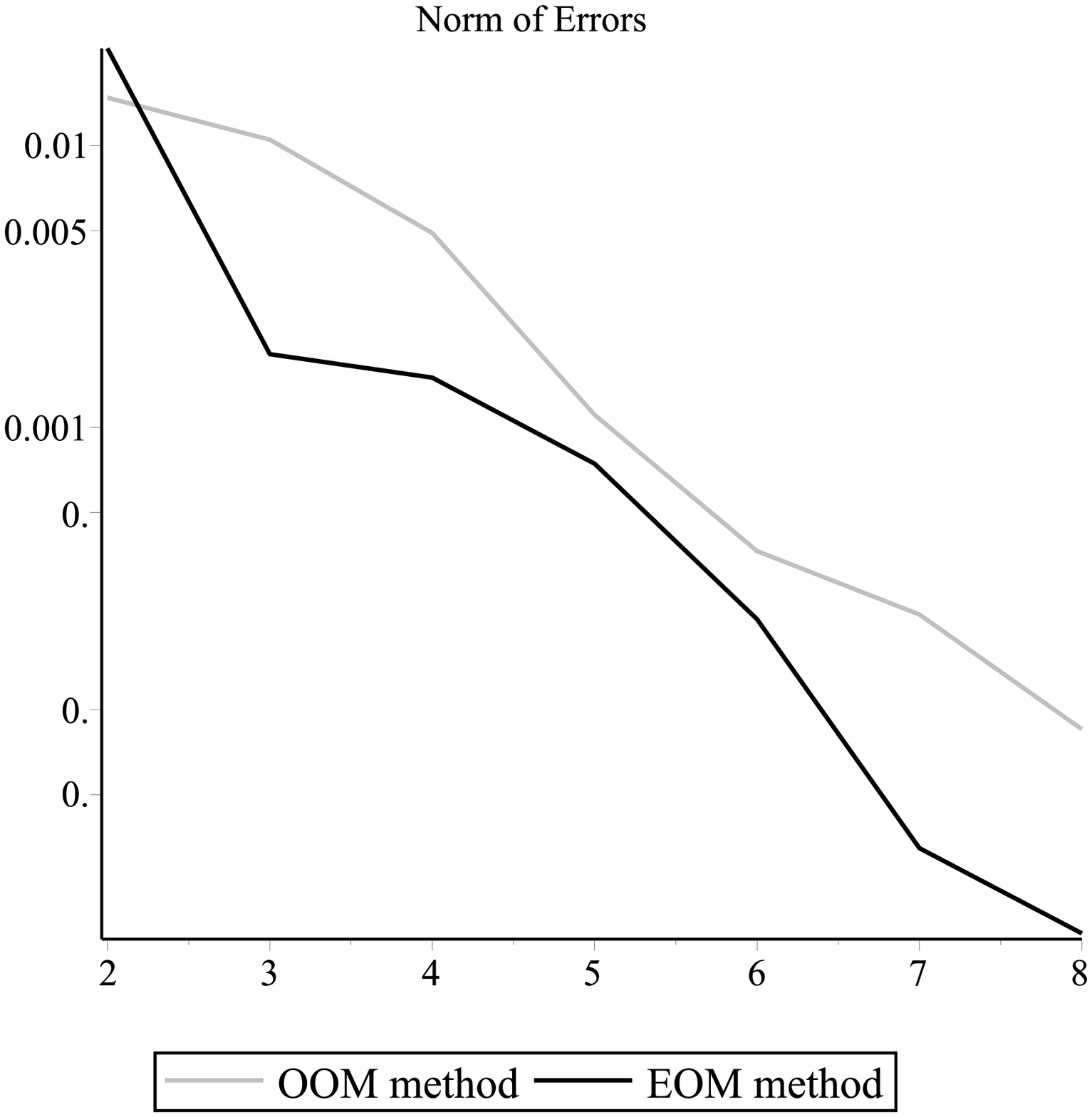}
\label{Fig_sXM5_8_8_errs}
}
\subfigure[$\fNorm{Residual\left(y_m(x)\right)}_1$]{
\centering
\includegraphics[scale=0.4]{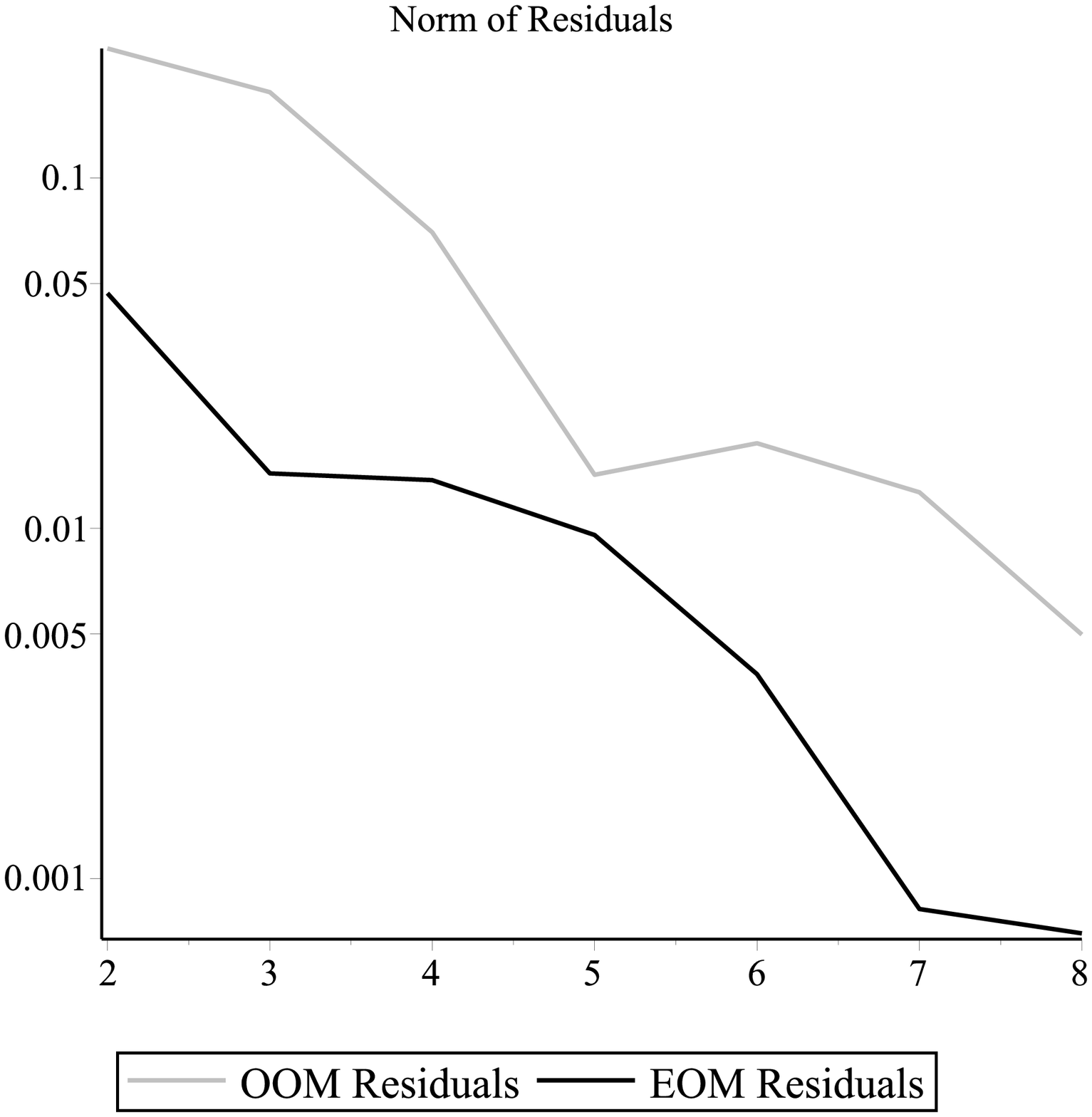}
\label{Fig_sXM5_8_8_ress}
}
\caption{
the norm1 of the residual and error function plots for several $m$ values and $f(x)= 1,~g(x)= y^5(x),~N=8$}
\label{Fig_sXM5_8_8_s}
\end{figure}
\begin{figure}
\centering
\subfigure[$\fNorm{Error\left(y_M(x)\right)}_1$]{
\centering
\includegraphics[scale=0.4]{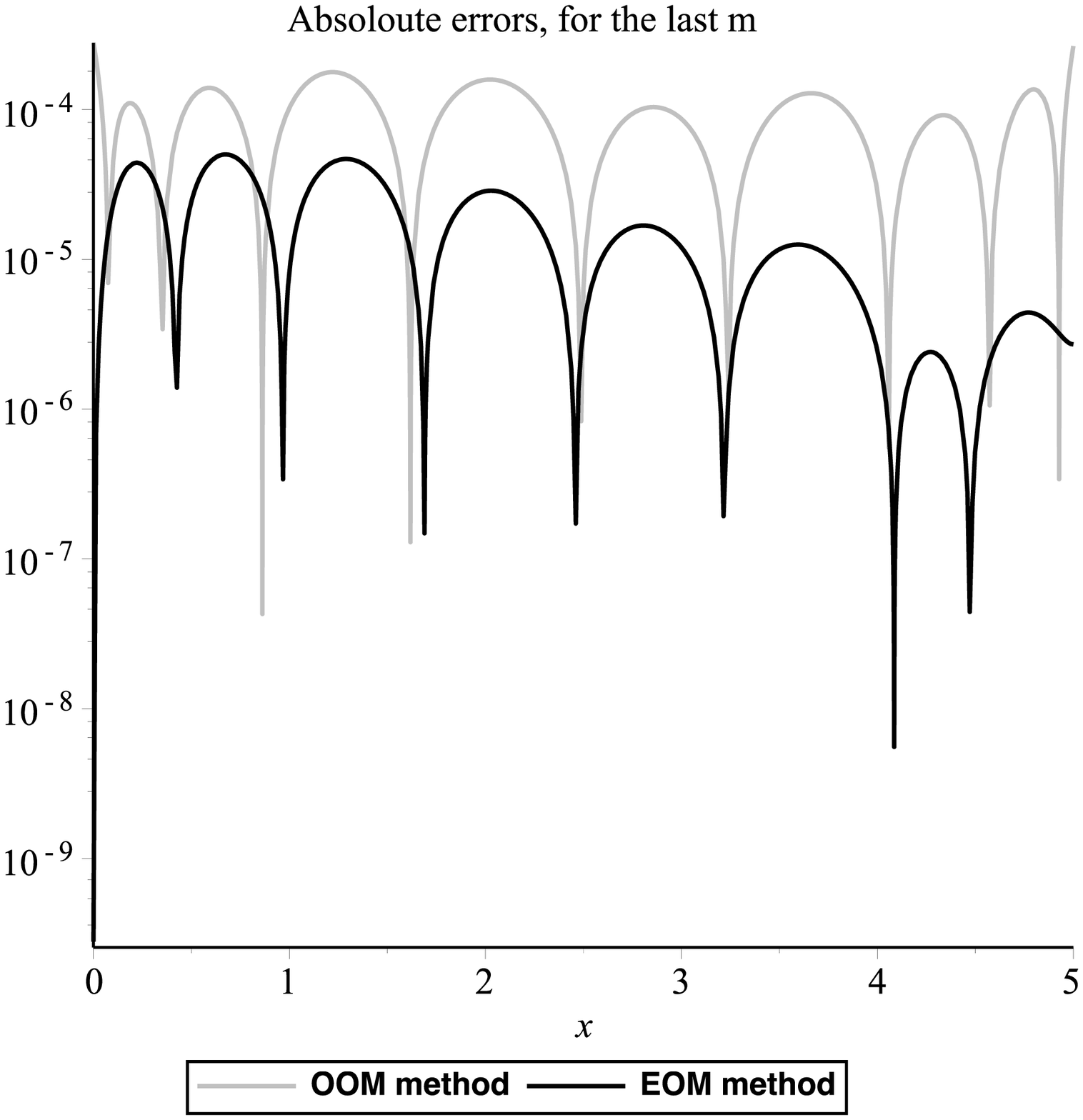}
\label{Fig_sXM5_8_8_last_err}
}
\subfigure[$\fNorm{Residual\left(y_M(x)\right)}_1$]{
\centering
\includegraphics[scale=0.4]{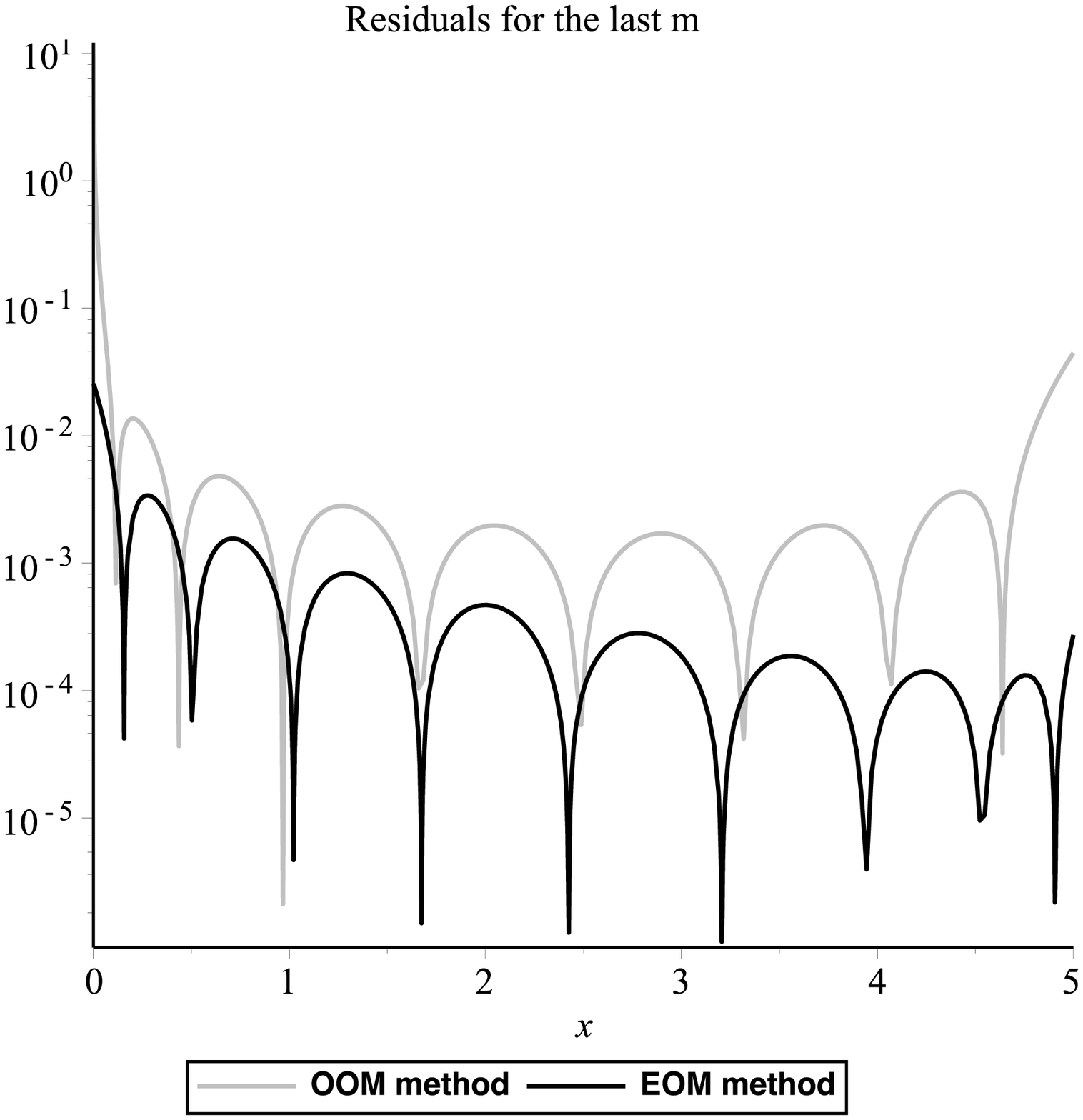}
\label{Fig_sXM5_8_8_last_res}
}
\caption{
the norm1 of the residual and error function plots for the largest $m$ value ($E$) of the figure (\ref{Fig_sXM5_8_8_s}) for $f(x)= 1,~g(x)= y^5(x),~N=8$}
\label{Fig_sXM5_8_8_last}
\end{figure}
\begin{figure}
\centering
\subfigure[$\fNorm{Error\left(y_m(x)\right)}_1$]{
\centering
\includegraphics[scale=0.4]{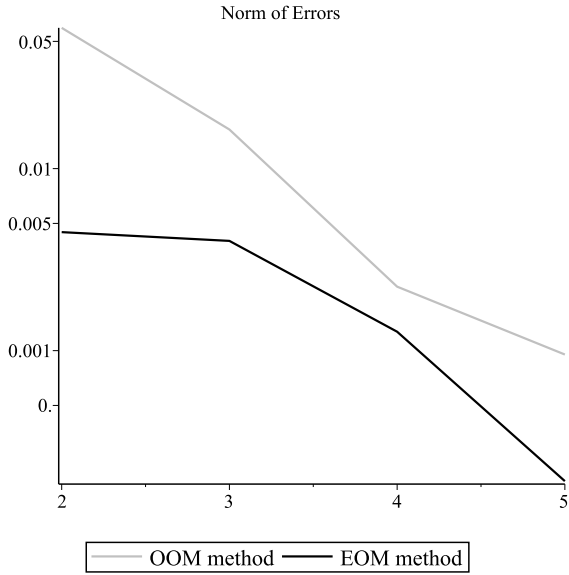}
\label{Fig_42exp_5_10_errs}
}
\subfigure[$\fNorm{Residual\left(y_m(x)\right)}_1$]{
\centering
\includegraphics[scale=0.4]{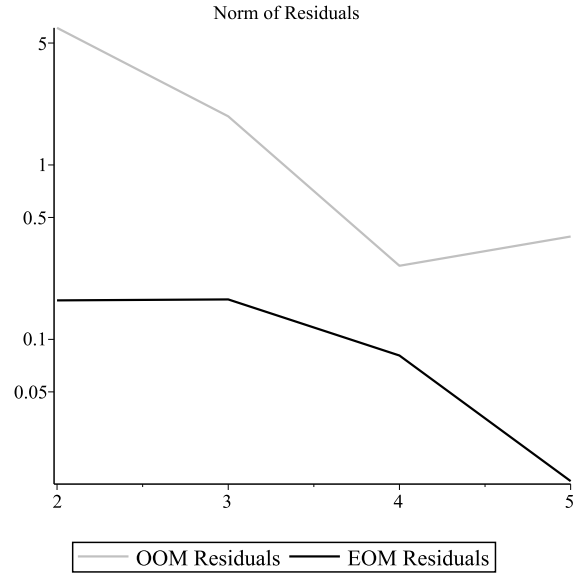}
\label{Fig_42exp_5_10_ress}
}
\caption{
the norm1 of the residual and error function plots for several $m$ values a $f(x)= 1,~g(x)= 4\left(2e^{y(x)}+ e^{\frac{y(x)}{2}}\right),~N=10$}
\label{Fig_42exp_5_10_s}
\end{figure}
\begin{figure}
\centering
\subfigure[$\fNorm{Error\left(y_M(x)\right)}_1$]{
\centering
\includegraphics[scale=0.4]{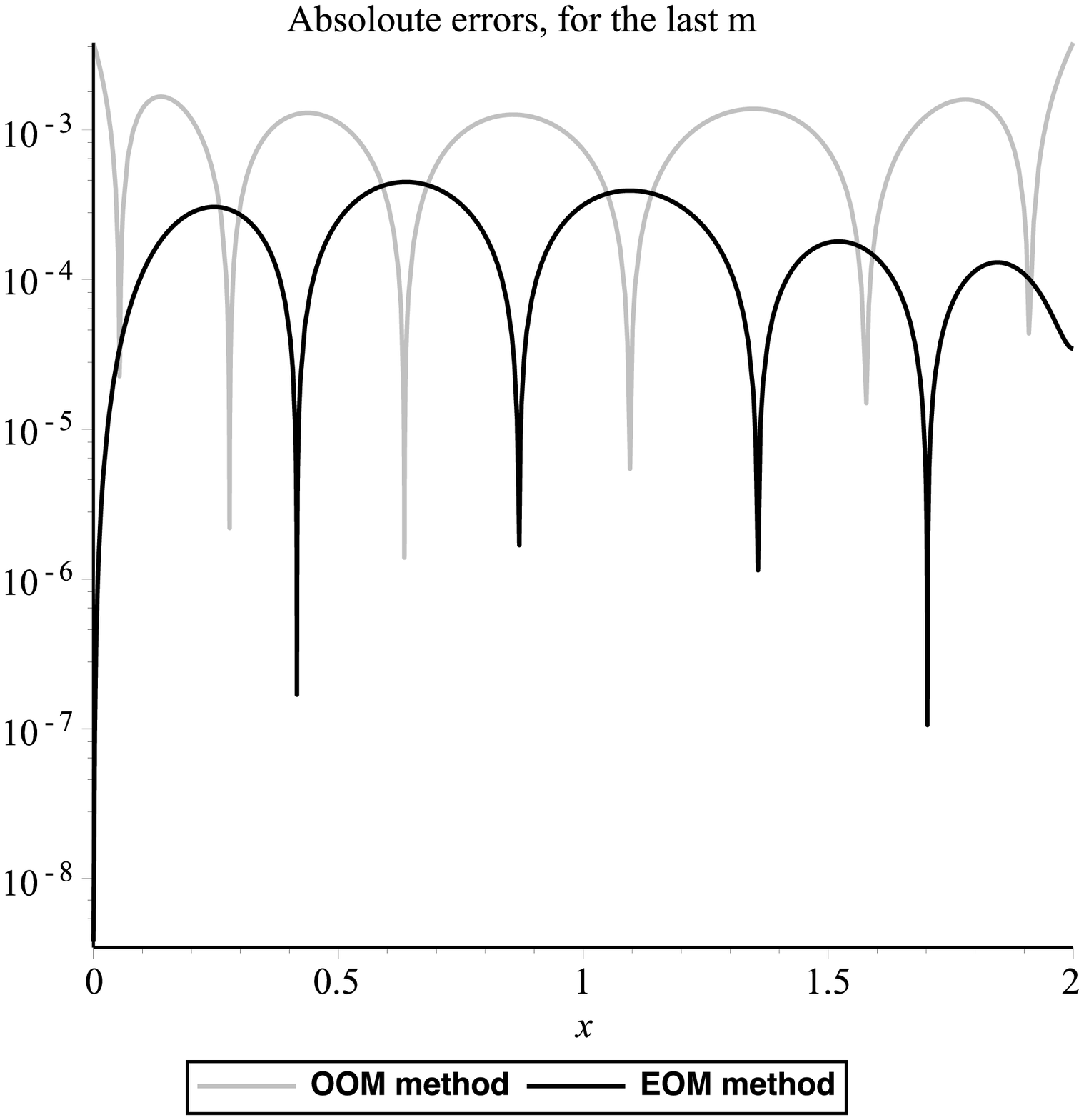}
\label{Fig_42exp_5_10_last_err}
}
\subfigure[$\fNorm{Residual\left(y_M(x)\right)}_1$]{
\centering
\includegraphics[scale=0.4]{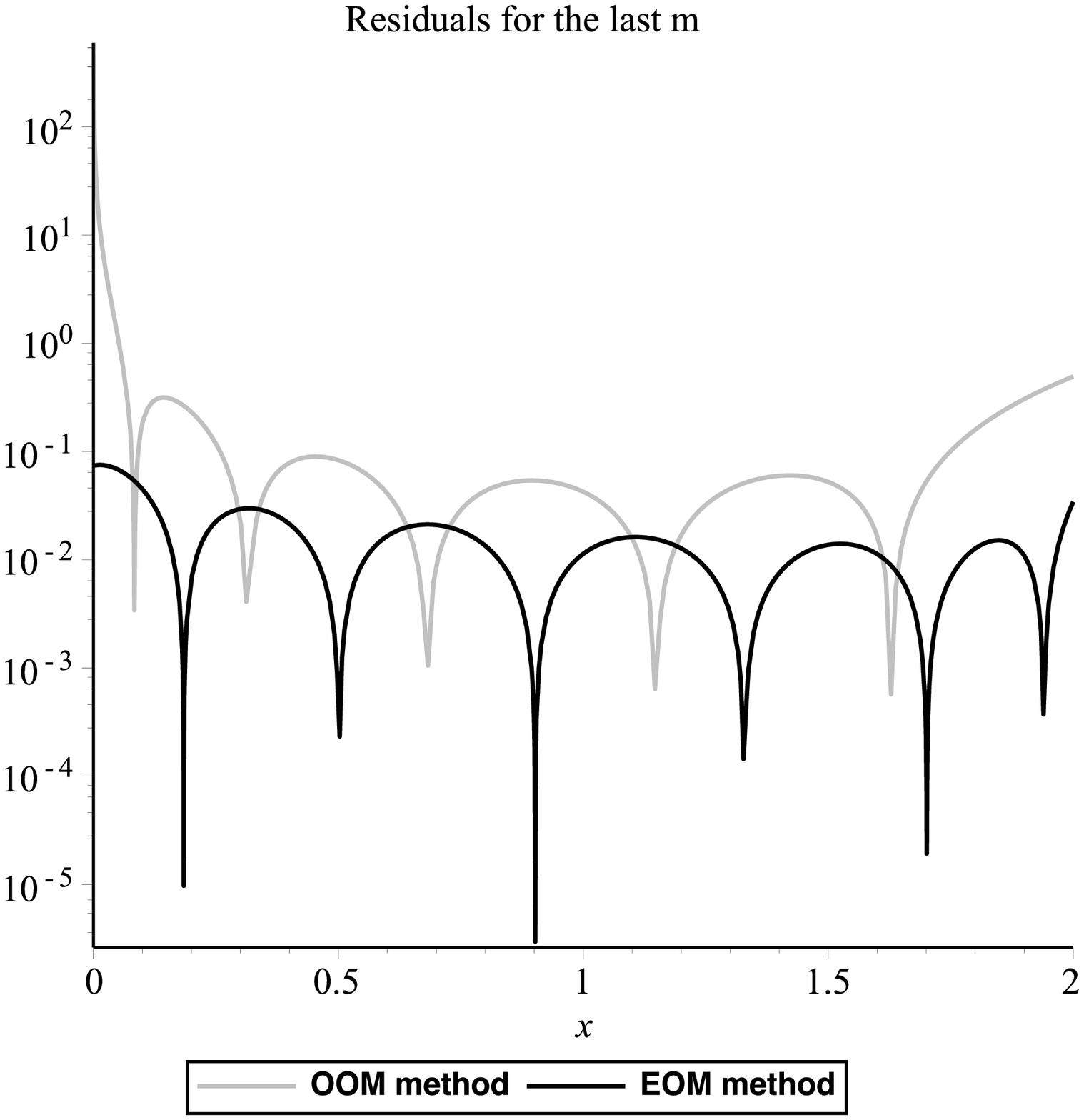}
\label{Fig_42exp_5_10_last_res}
}
\caption{
the norm1 of the residual and error function plots for the largest $m$ value ($E$) of the figure (\ref{Fig_42exp_5_10_s}) for $f(x)= 1,~g(x)= 4\left(2e^{y(x)}+ e^{\frac{y(x)}{2}},~N=10\right)$}
\label{Fig_42exp_5_10_last}
\end{figure}
\begin{figure}
\centering
\subfigure[$\fNorm{Error\left(y_m(x)\right)}_1$]{
\centering
\includegraphics[scale=0.4]{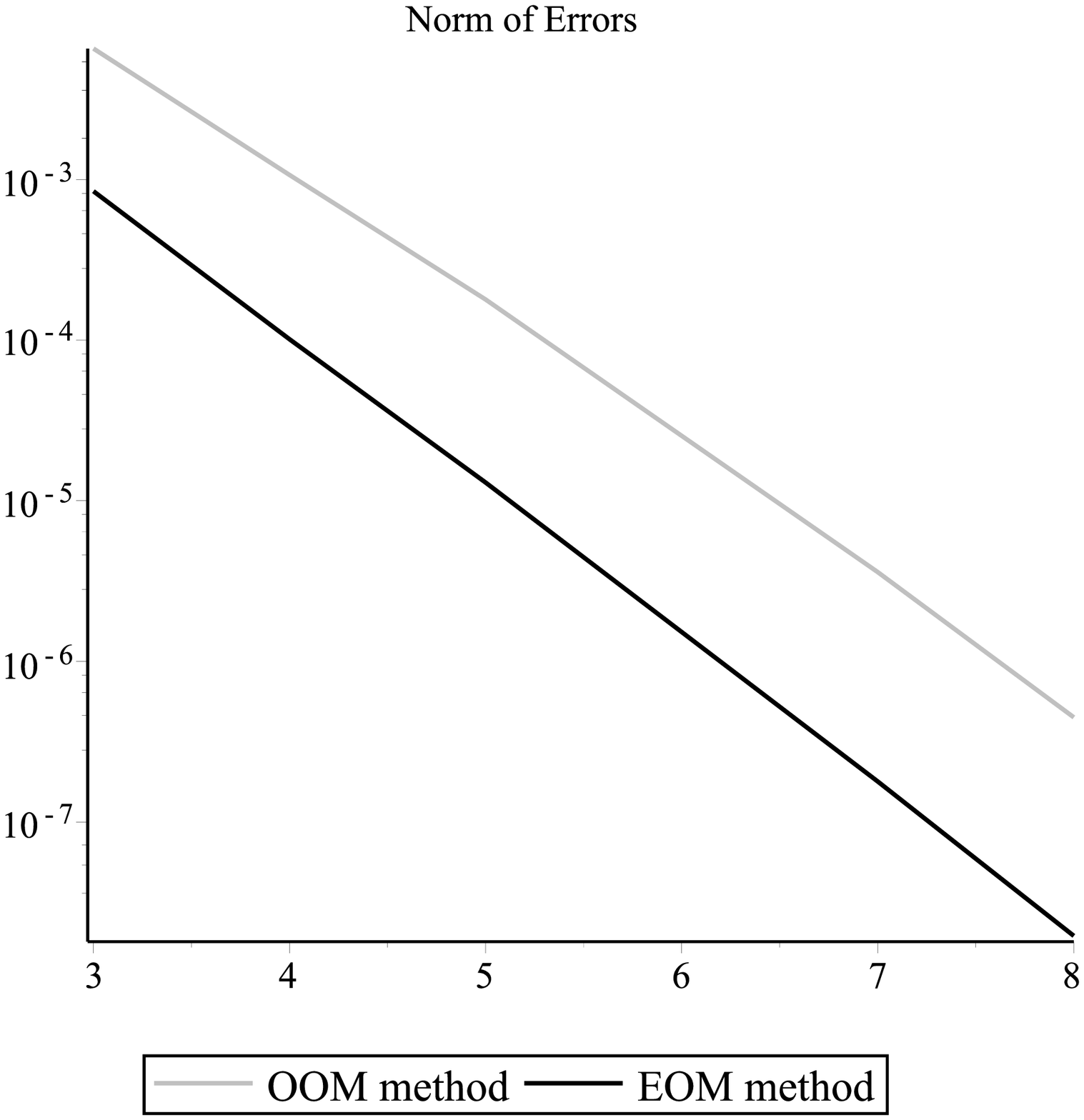}
\label{Fig_2LM_8_8_errs}
}
\subfigure[$\fNorm{Residual\left(y_m(x)\right)}_1$]{
\centering
\includegraphics[scale=0.4]{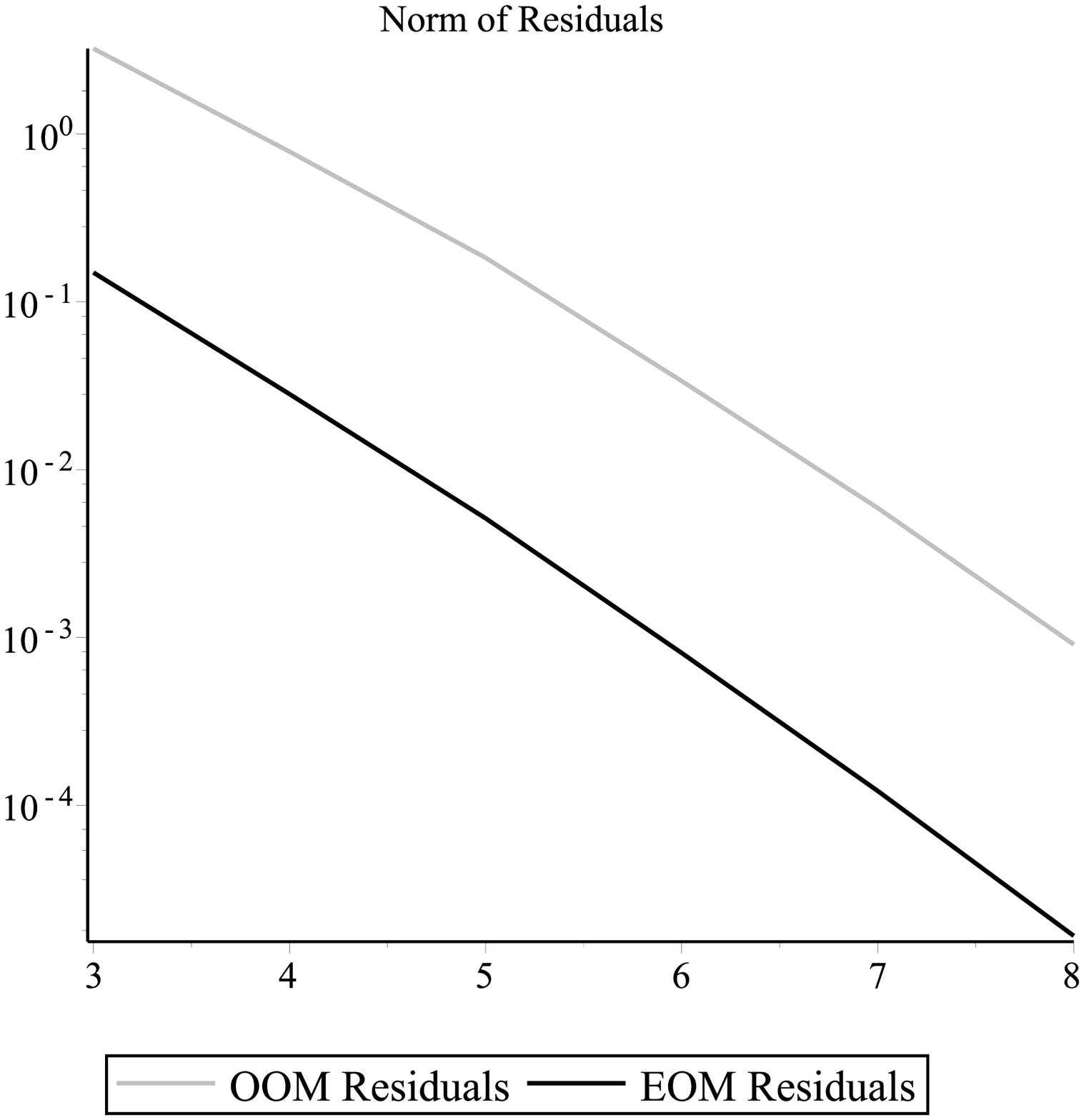}
\label{Fig_2LM_8_8_ress}
}
\caption{
the norm1 of the residual and error function plots for several $m$ values and $f(x)= -2\left(2x^2+ 3\right),~g(x)= y(x),~N=8$}
\label{Fig_2LM_8_8_s}
\end{figure}
\begin{figure}
\centering
\subfigure[$\fNorm{Error\left(y_M(x)\right)}_1$]{
\centering
\includegraphics[scale=0.4]{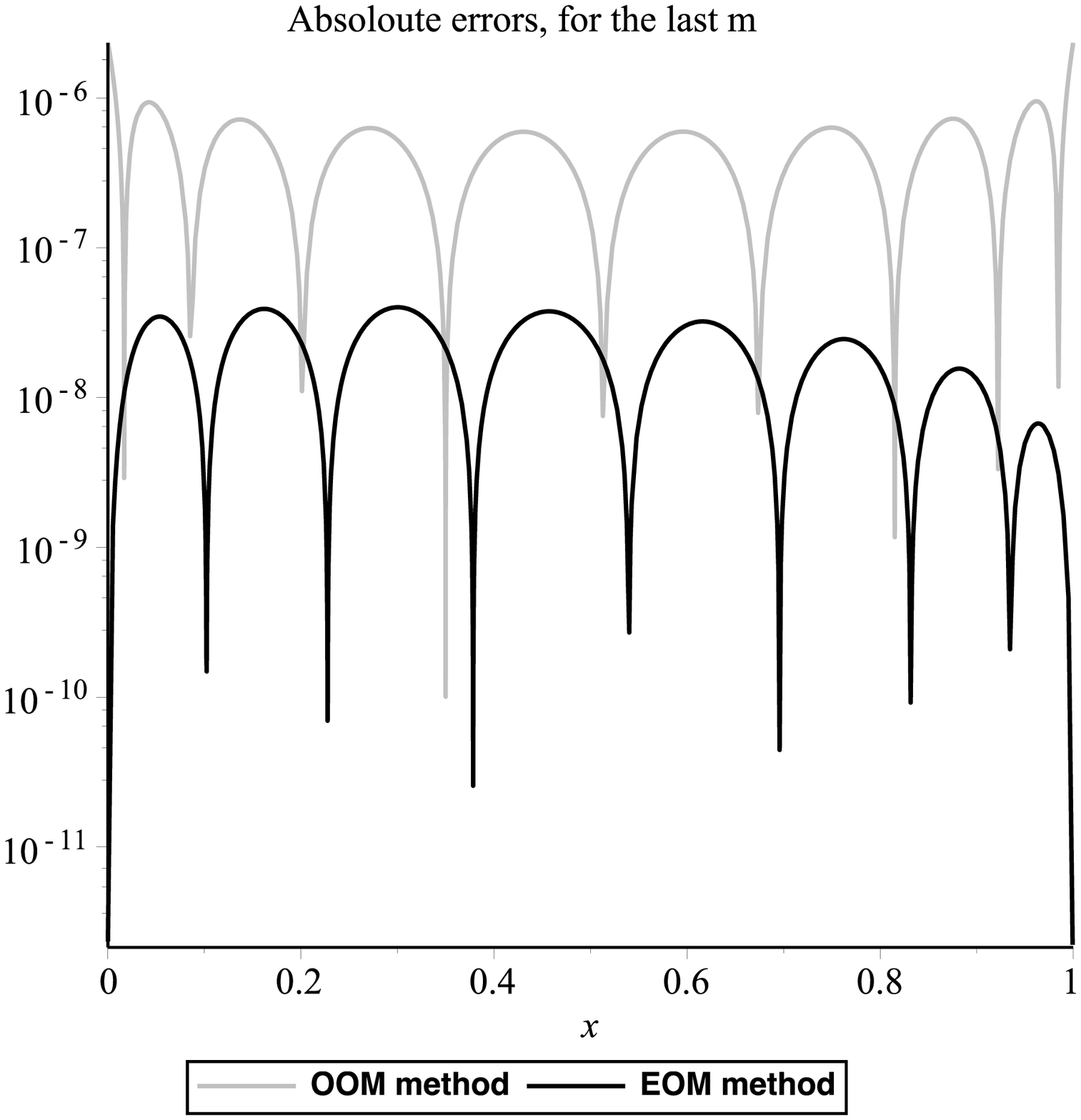}
\label{Fig_2LM_8_8_last_err}
}
\subfigure[$\fNorm{Residual\left(y_M(x)\right)}_1$]{
\centering
\includegraphics[scale=0.4]{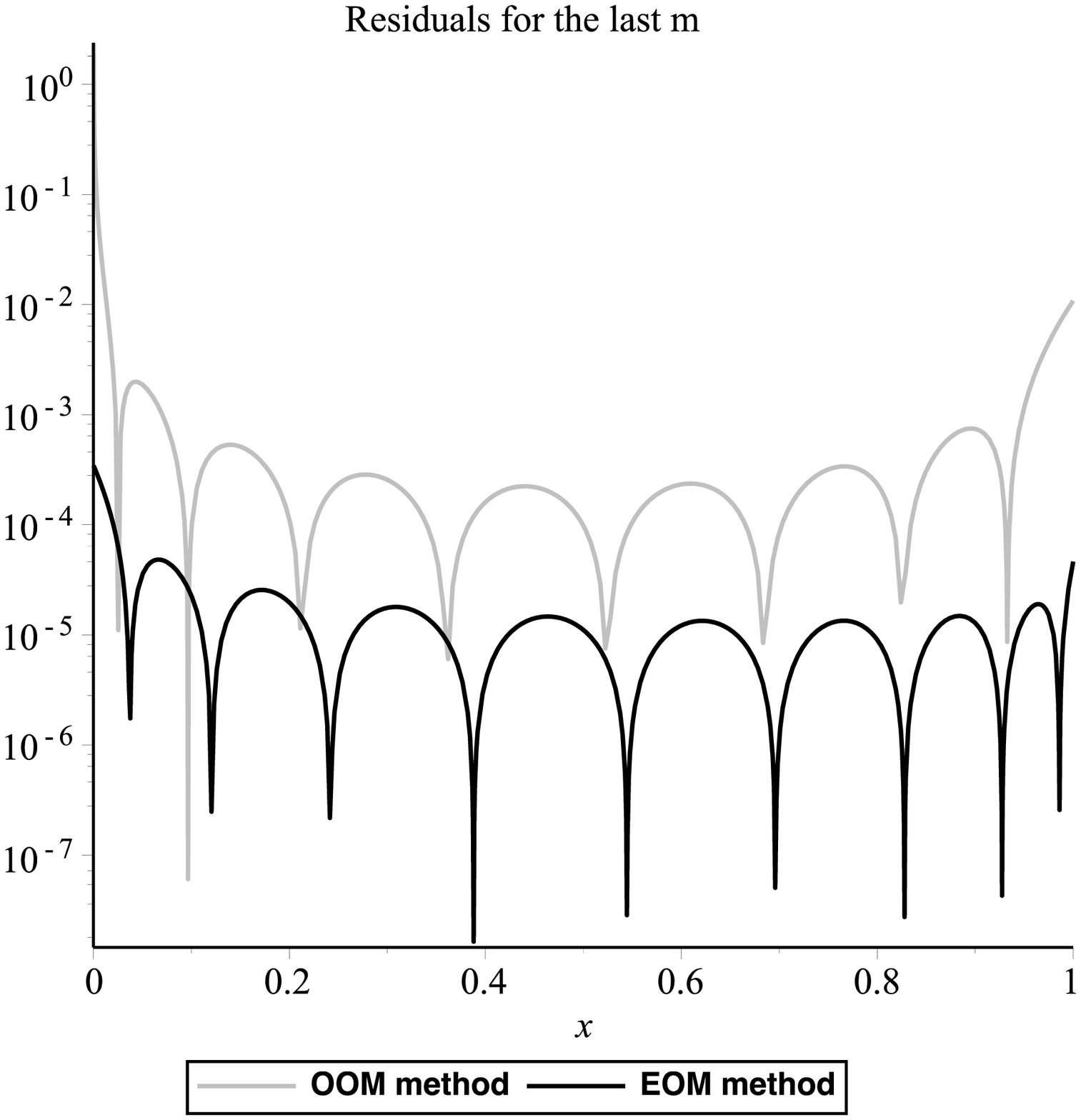}
\label{Fig_2LM_8_8_last_res}
}
\caption{
the norm1 of the residual and error function plots for the largest $m$ value ($E$) of the figure
(\ref{Fig_2LM_8_8_s}) for $f(x)= -2\left(2x^2+ 3\right),~g(x)= y(x),~N=8$}
\label{Fig_2LM_8_8_last}
\end{figure}
\begin{figure}
\centering
\subfigure[$\fNorm{Error\left(y_m(x)\right)}_1$]{
\centering
\includegraphics[scale=0.4]{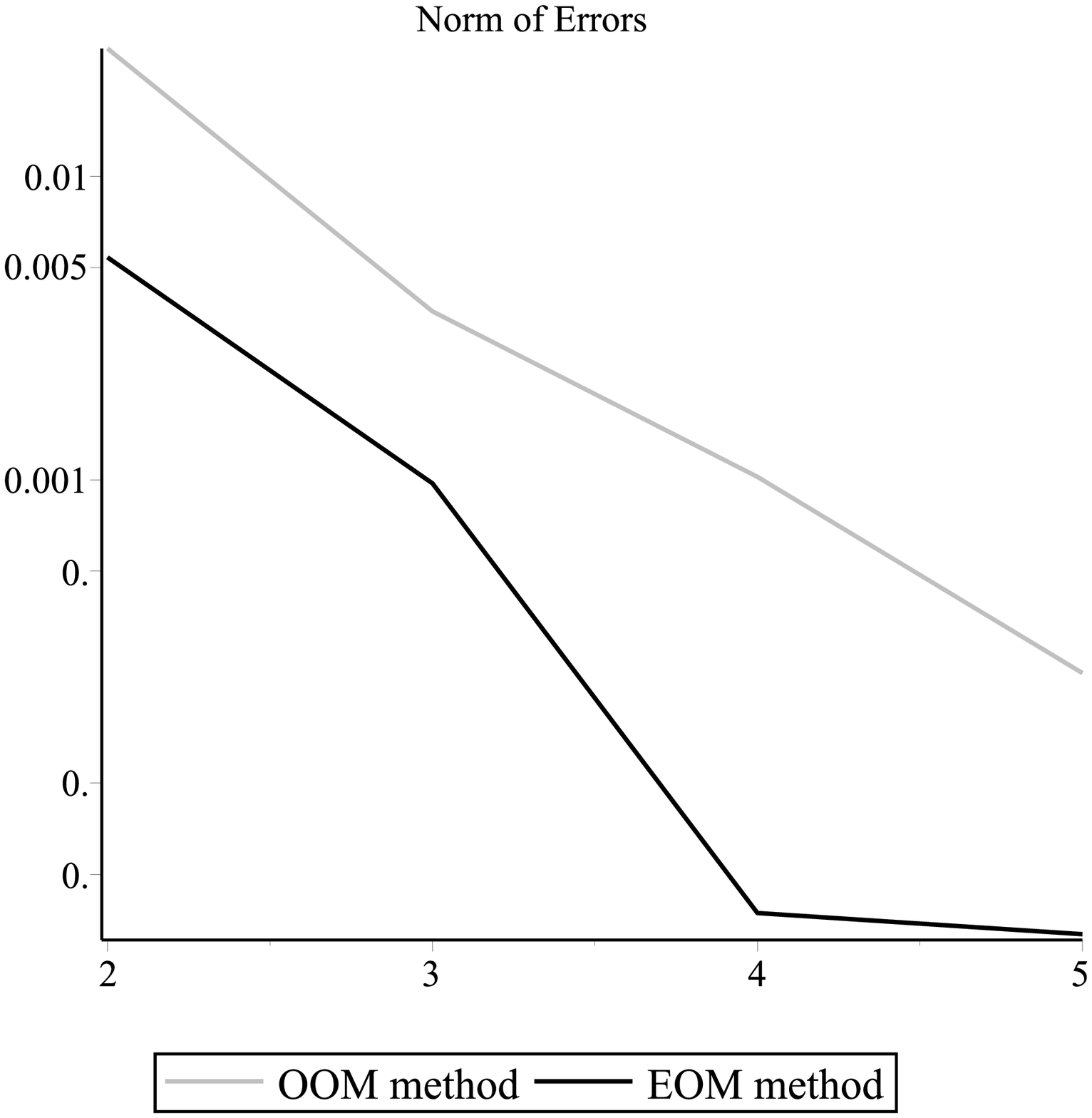}
\label{Fig_xM1A5_5_12_errs}
}
\subfigure[$\fNorm{Residual\left(y_m(x)\right)}_1$]{
\centering
\includegraphics[scale=0.4]{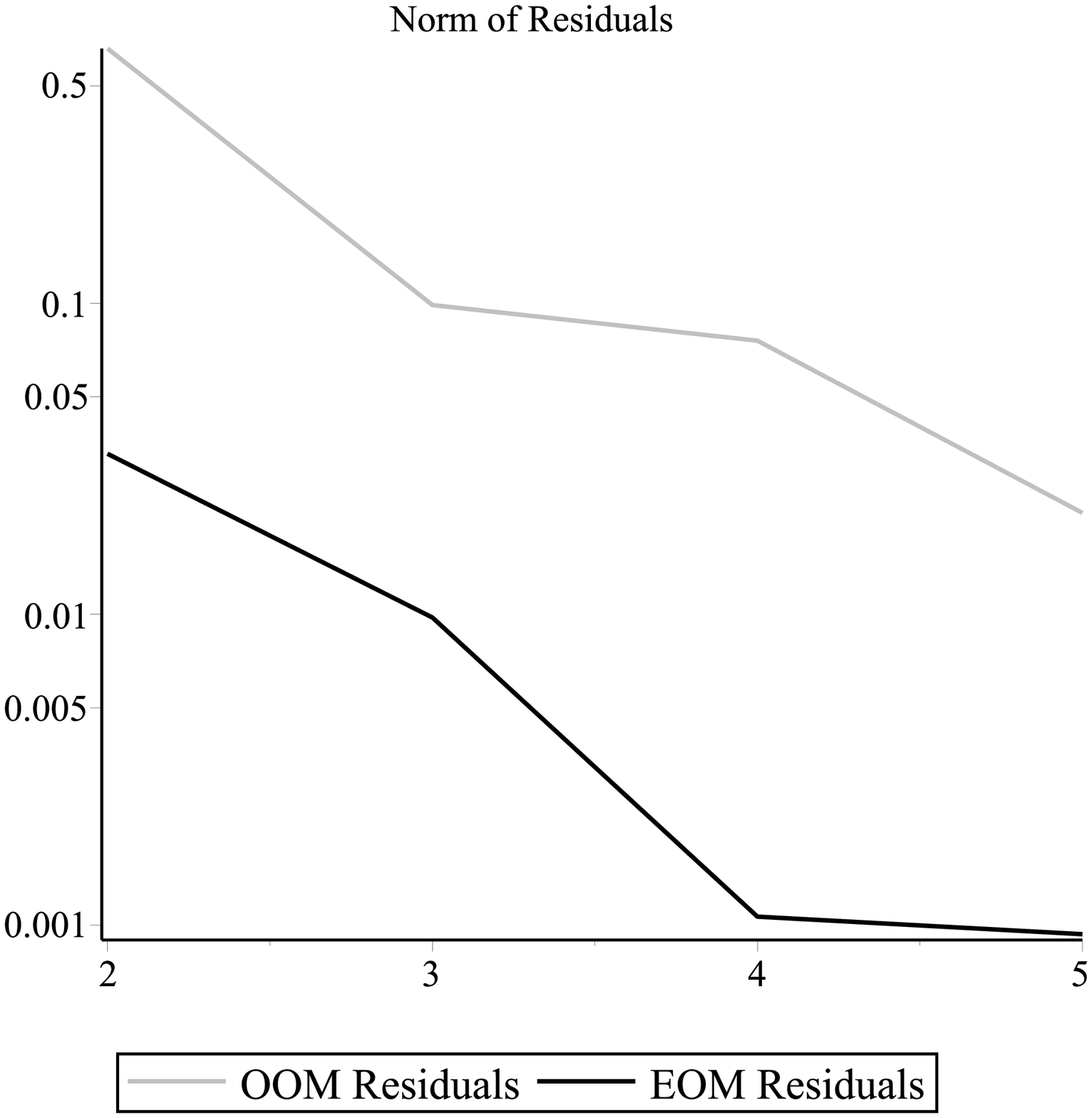}
\label{Fig_xM1A5_5_12_ress}
}
\caption{
the norm1 of the residual and error function plots for several $m$ values and $f(x)= 1,~g(x)= y^{\frac{3}{2}}(x),~N=12$}
\label{Fig_xM1A5_5_12_s}
\end{figure}
\begin{figure}
\centering
\subfigure[$\fNorm{Error\left(y_M(x)\right)}_1$]{
\centering
\includegraphics[scale=0.4]{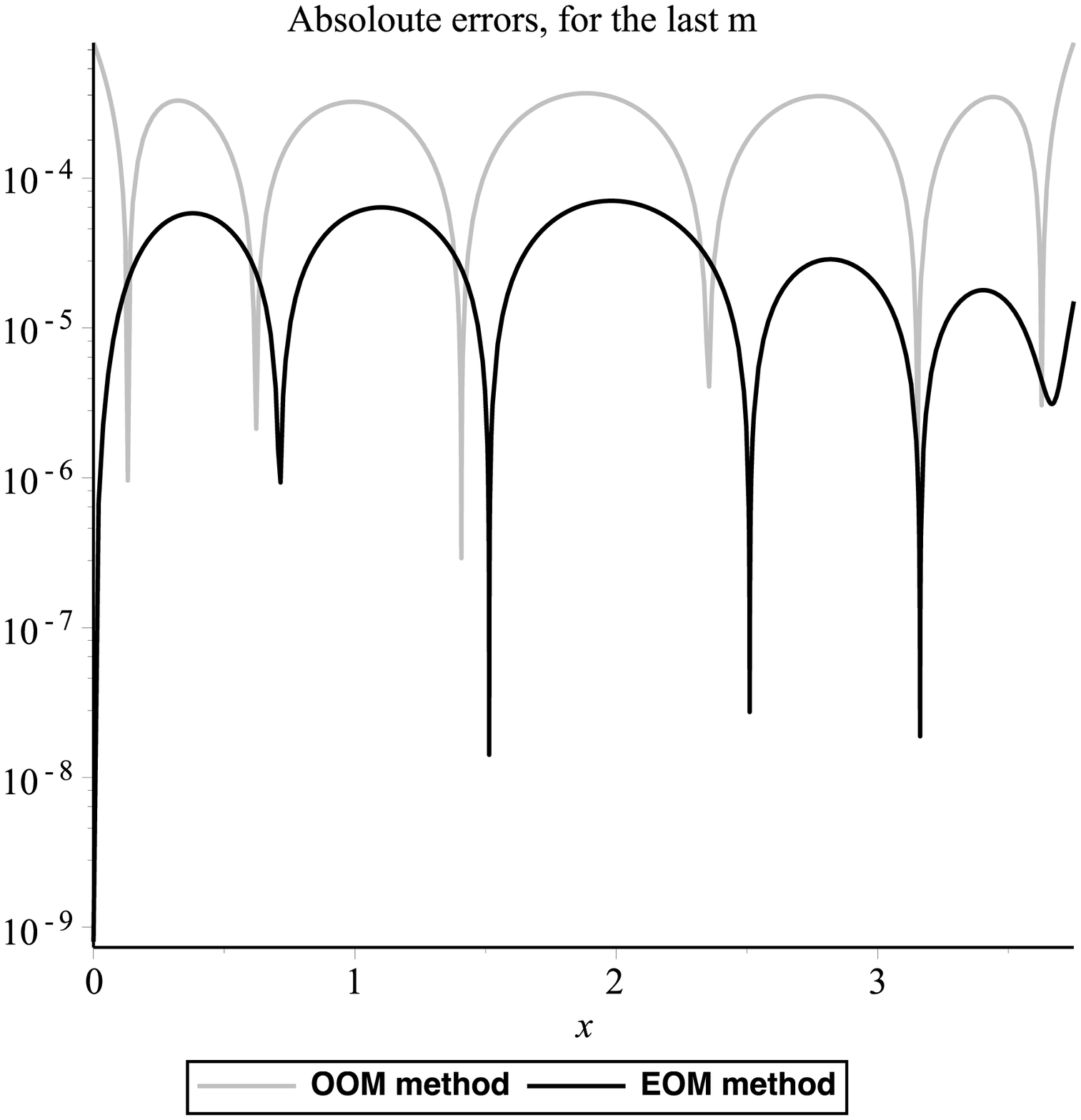}
\label{Fig_xM1A5_5_12_last_err}
}
\subfigure[$\fNorm{Residual\left(y_M(x)\right)}_1$]{
\centering
\includegraphics[scale=0.4]{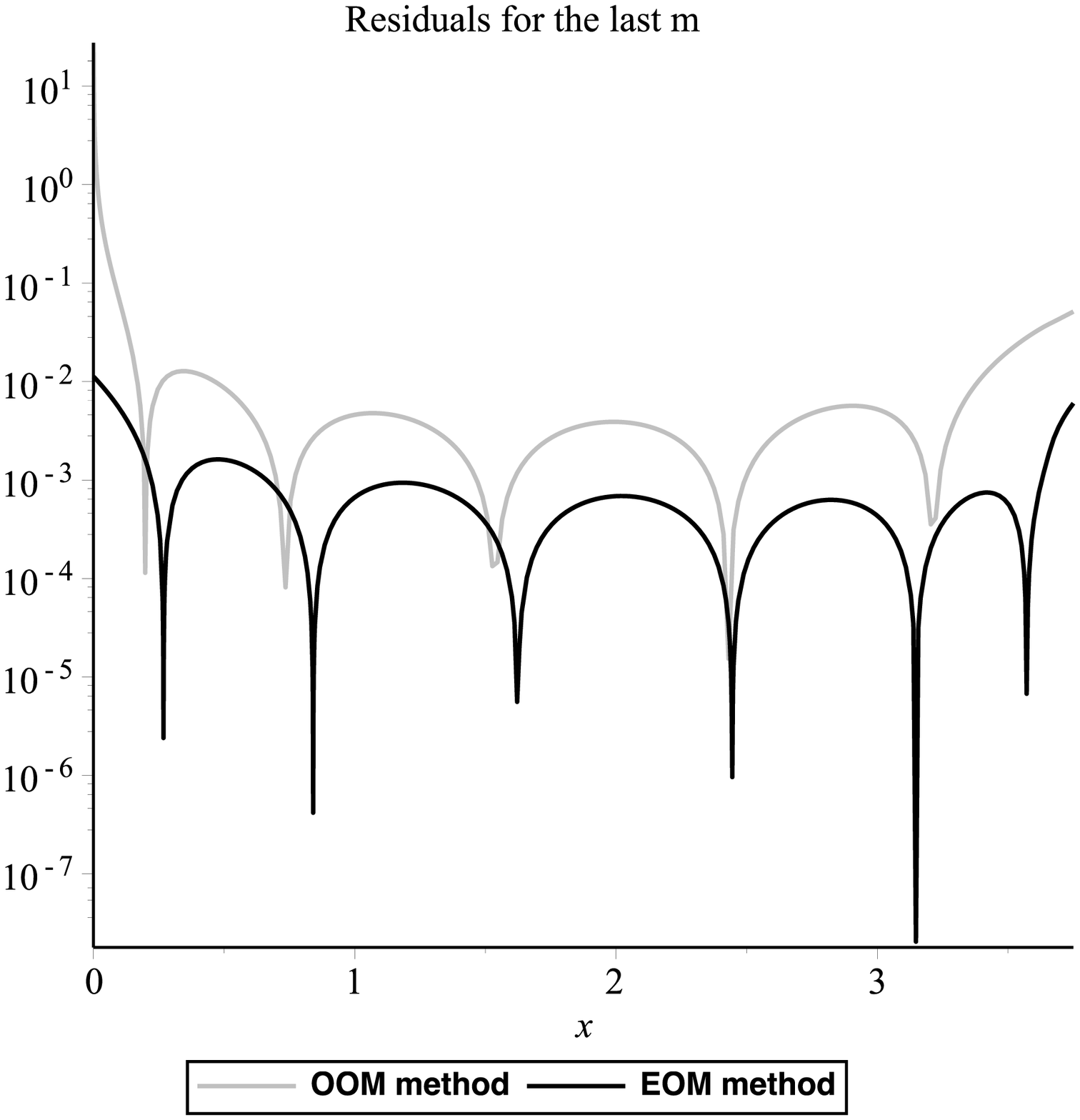}
\label{Fig_xM1A5_5_12_last_res}
}
\caption{
the norm1 of the residual and error function plots for the largest $m$ value ($E$) of the figure
(\ref{Fig_xM1A5_5_12_s}) for $f(x)= 1,~g(x)= y^{\frac{3}{2}}(x),~N=12$}
\label{Fig_xM1A5_5_12_last}
\end{figure}
\begin{figure}
\centering
\subfigure[$\fNorm{Error\left(y_m(x)\right)}_1$]{
\centering
\includegraphics[scale=0.4]{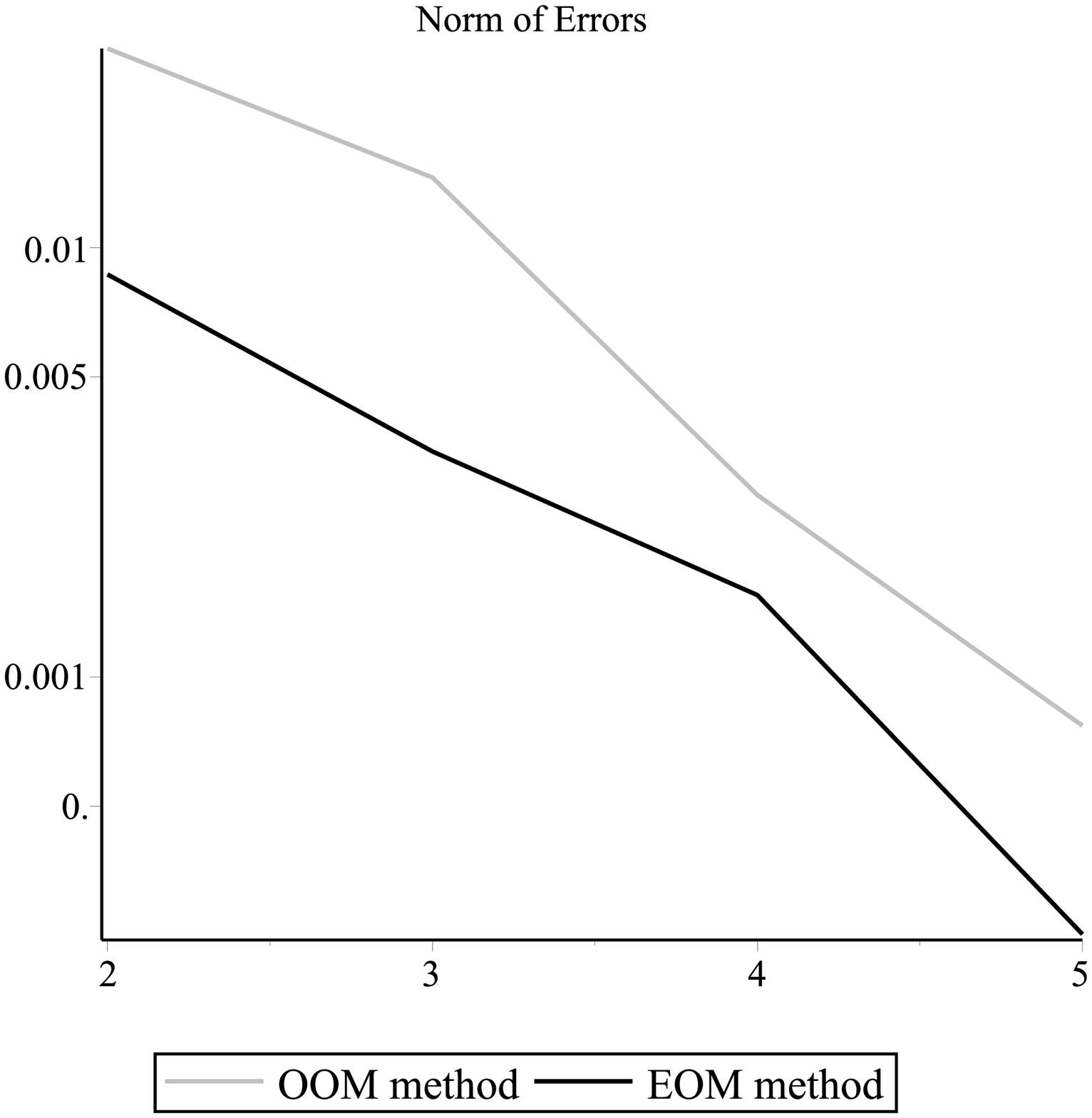}
\label{Fig_xM2A5_5_12_errs}
}
\subfigure[$\fNorm{Residual\left(y_m(x)\right)}_1$]{
\centering
\includegraphics[scale=0.4]{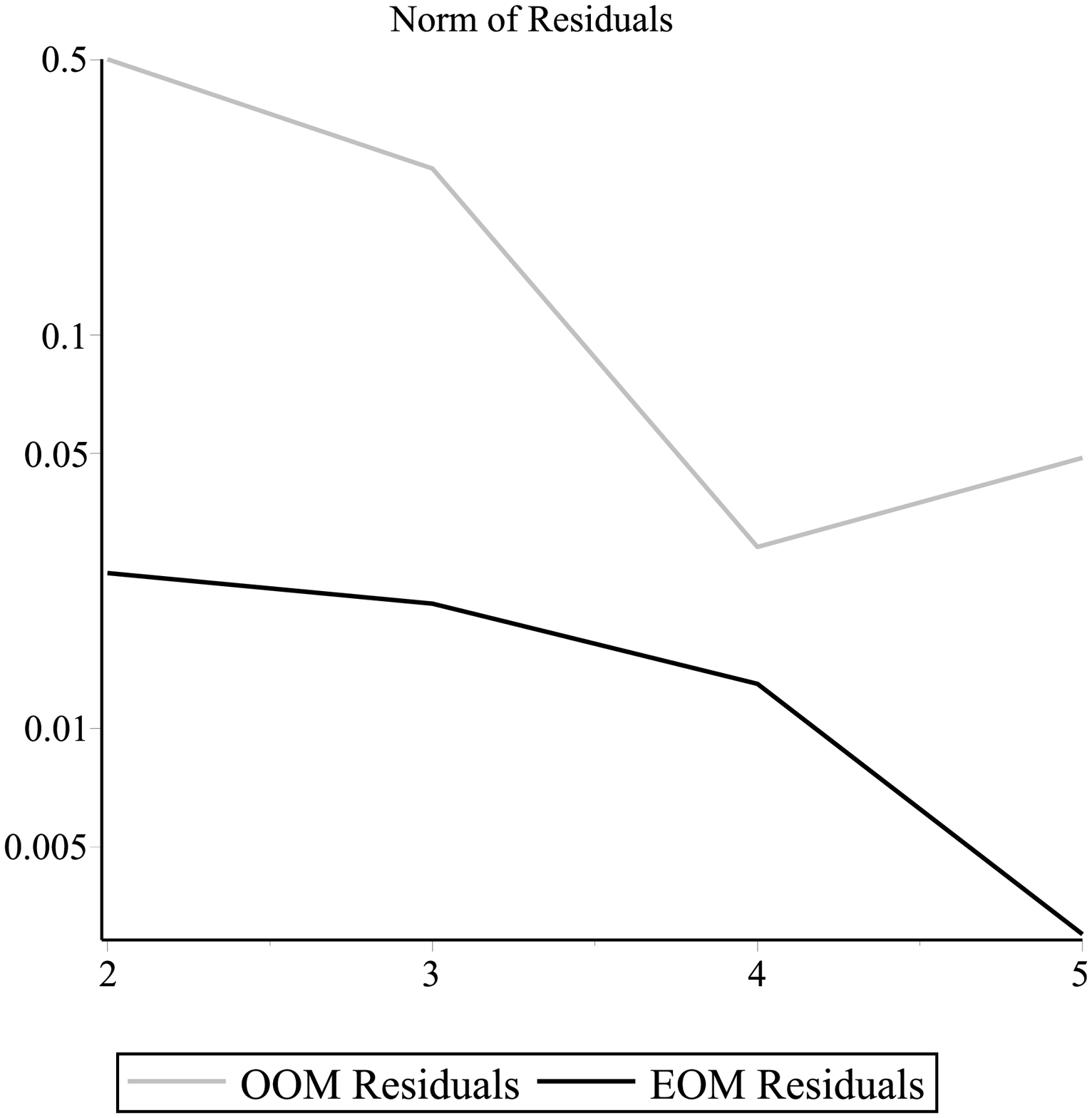}
\label{Fig_xM2A5_5_12_ress}
}
\caption{
the norm1 of the residual and error function plots for several $m$ values and $f(x)= 1,~g(x)= y^{\frac{5}{2}}(x),~N=12$}
\label{Fig_xM2A5_5_12_s}
\end{figure}
\begin{figure}
\centering
\subfigure[$\fNorm{Error\left(y_M(x)\right)}_1$]{
\centering
\includegraphics[scale=0.4]{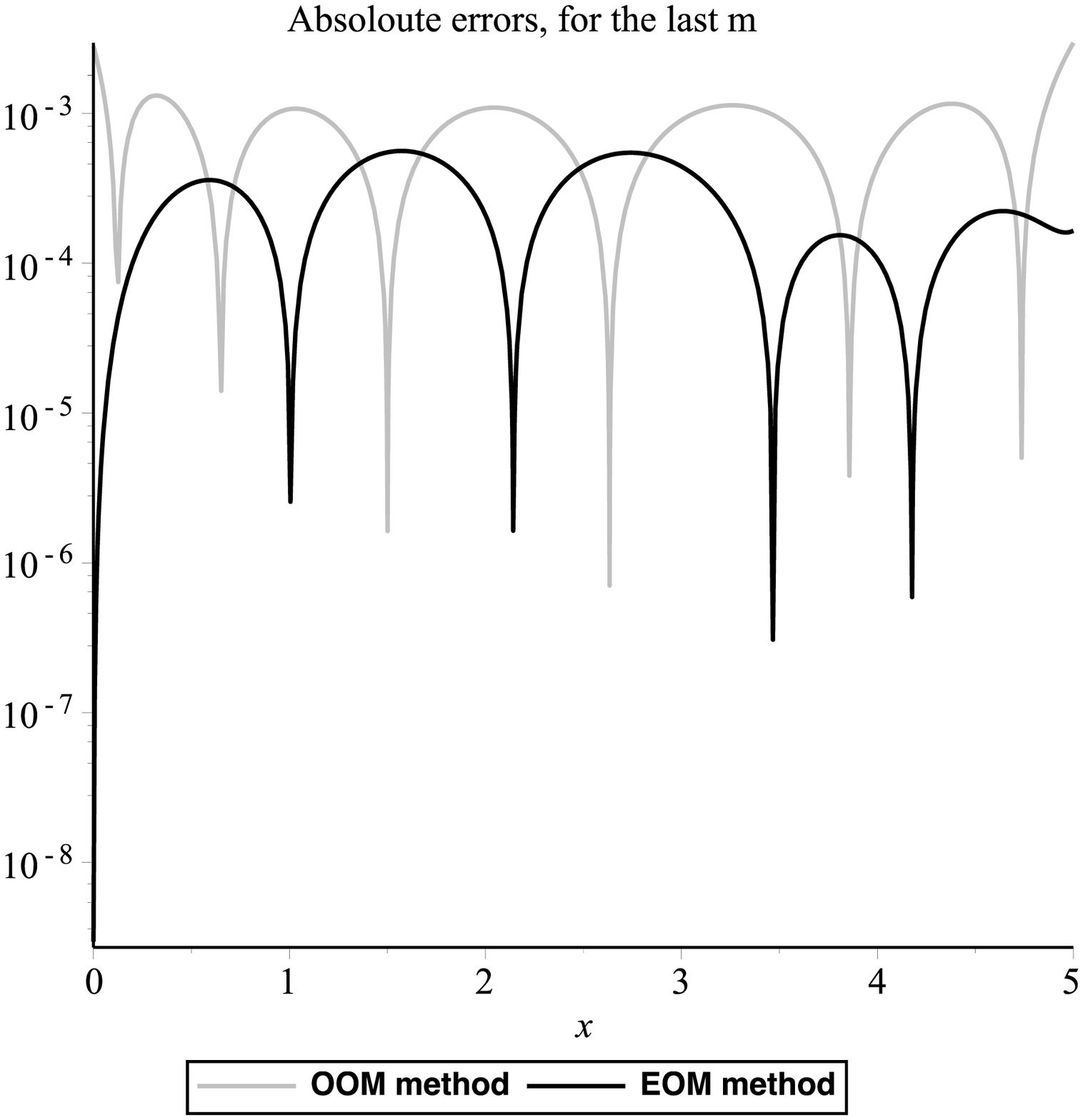}
\label{Fig_xM2A5_5_12_last_err}
}
\subfigure[$\fNorm{Residual\left(y_M(x)\right)}_1$]{
\centering
\includegraphics[scale=0.4]{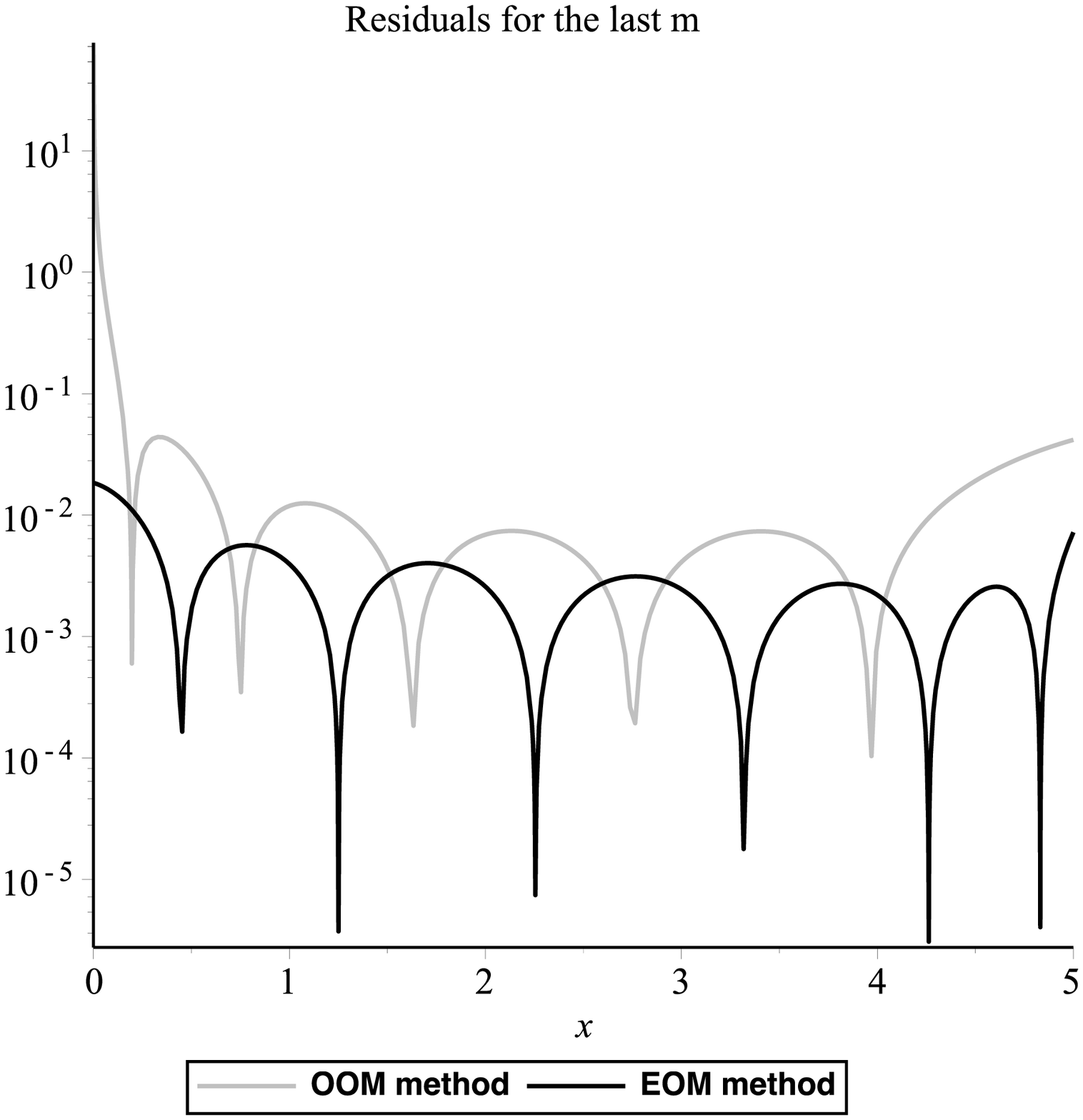}
\label{Fig_xM2A5_5_12_last_res}
}
\caption{
the norm1 of the residual and error function plots for the largest $m$ value ($E$) of the figure
(\ref{Fig_xM2A5_5_12_s}) for $f(x)= 1,~g(x)= y^{\frac{5}{2}}(x),~N=12$}
\label{Fig_xM2A5_5_12_last}
\end{figure}
\begin{figure}
\centering
\subfigure[$\fNorm{Error\left(y_M(x)\right)}_1$]{
\centering
\includegraphics[scale=0.4]{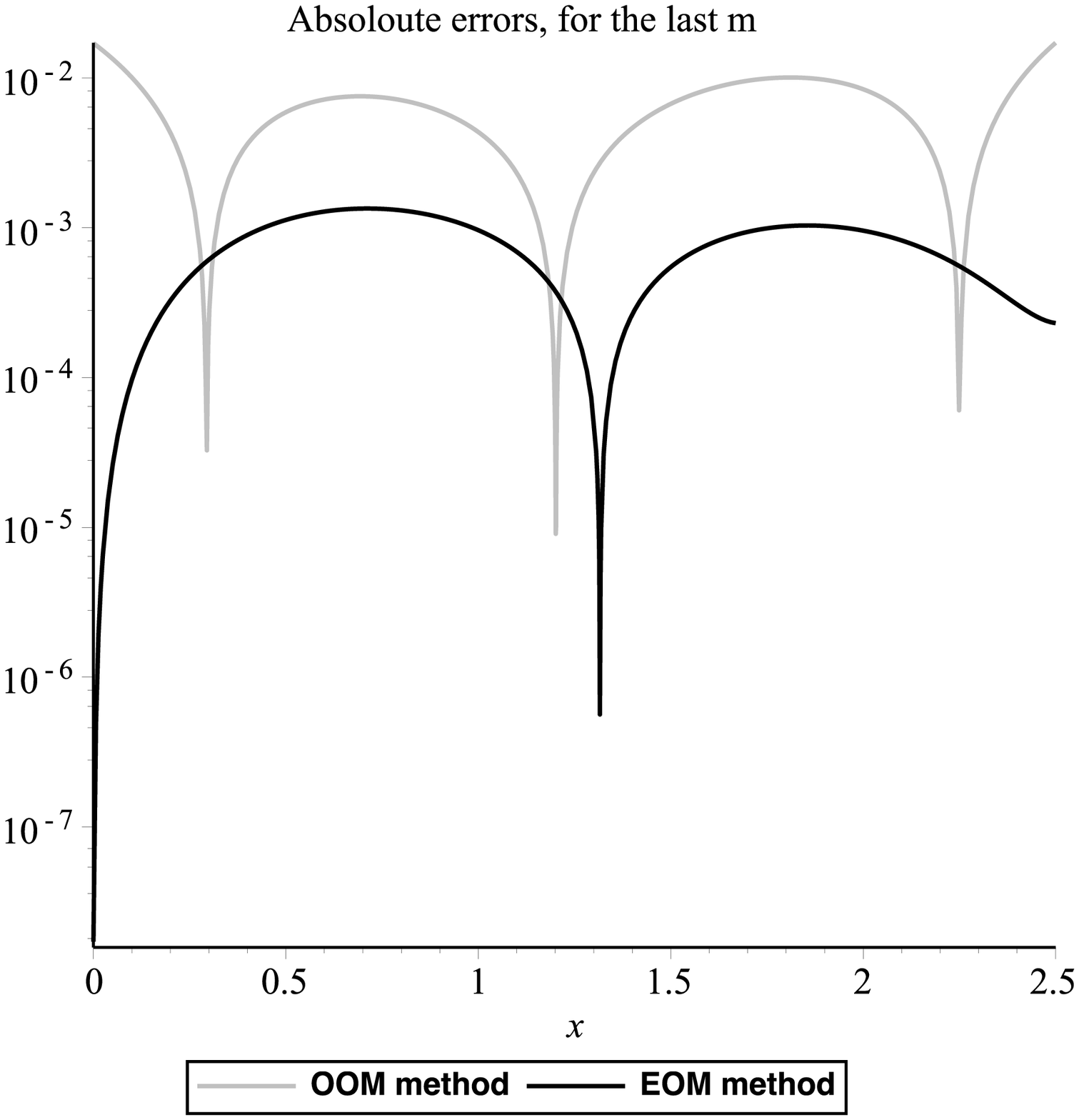}
\label{Fig_exp_res_2_15_last_err}
}
\subfigure[$\fNorm{Residual\left(y_M(x)\right)}_1$]{
\centering
\includegraphics[scale=0.4]{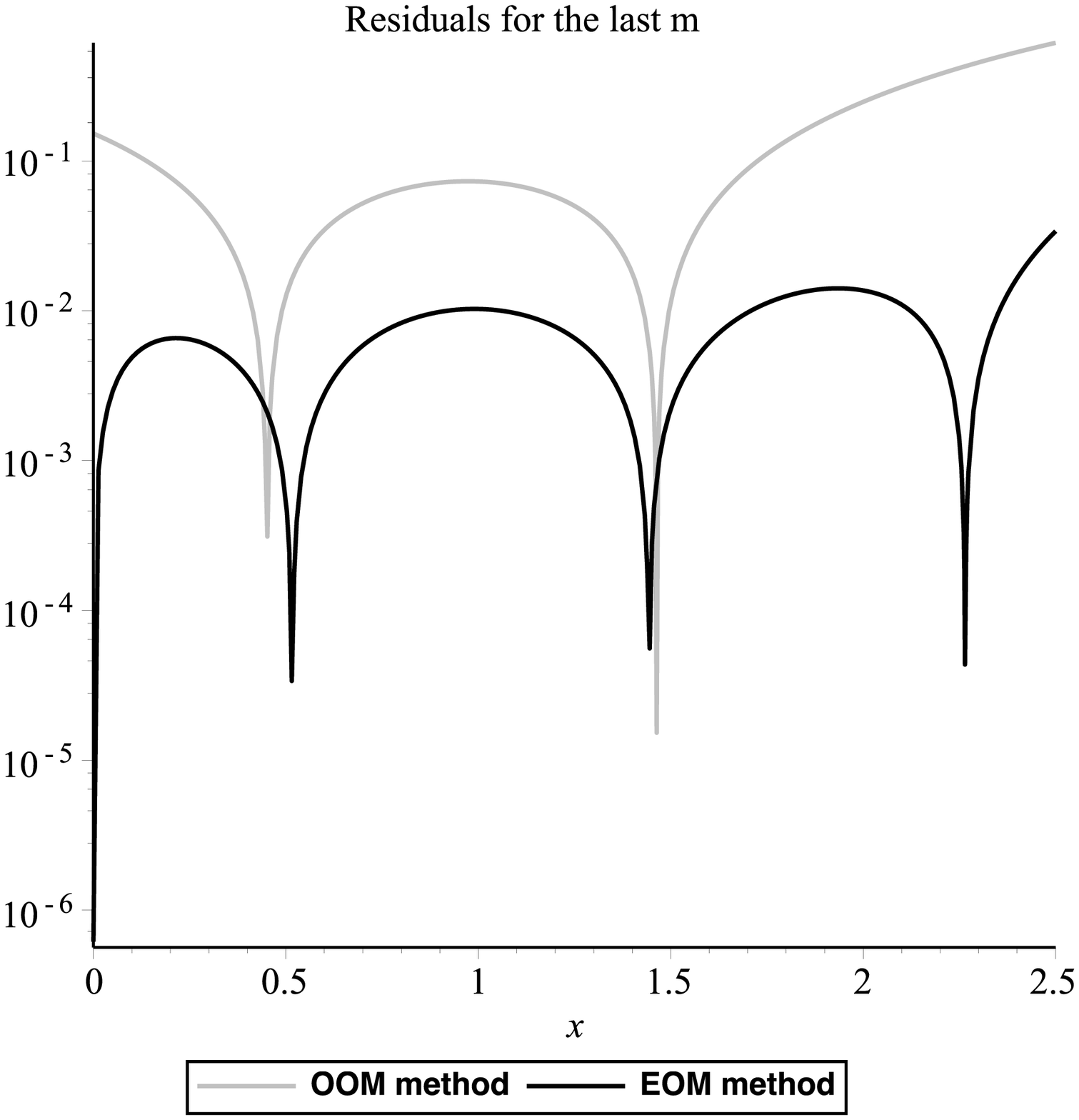}
\label{Fig_exp_res_2_15_last_res}
}
\caption{
the norm1 of the residual and error function for $M=2$ and $f(x)= 1,~g(x)= e^{y(x)},~N=15$}
\label{Fig_exp_res_2_15_last}
\end{figure}
\begin{figure}
\centering
\subfigure[$\fNorm{Error\left(y_m(x)\right)}_1$]{
\centering
\includegraphics[scale=0.4]{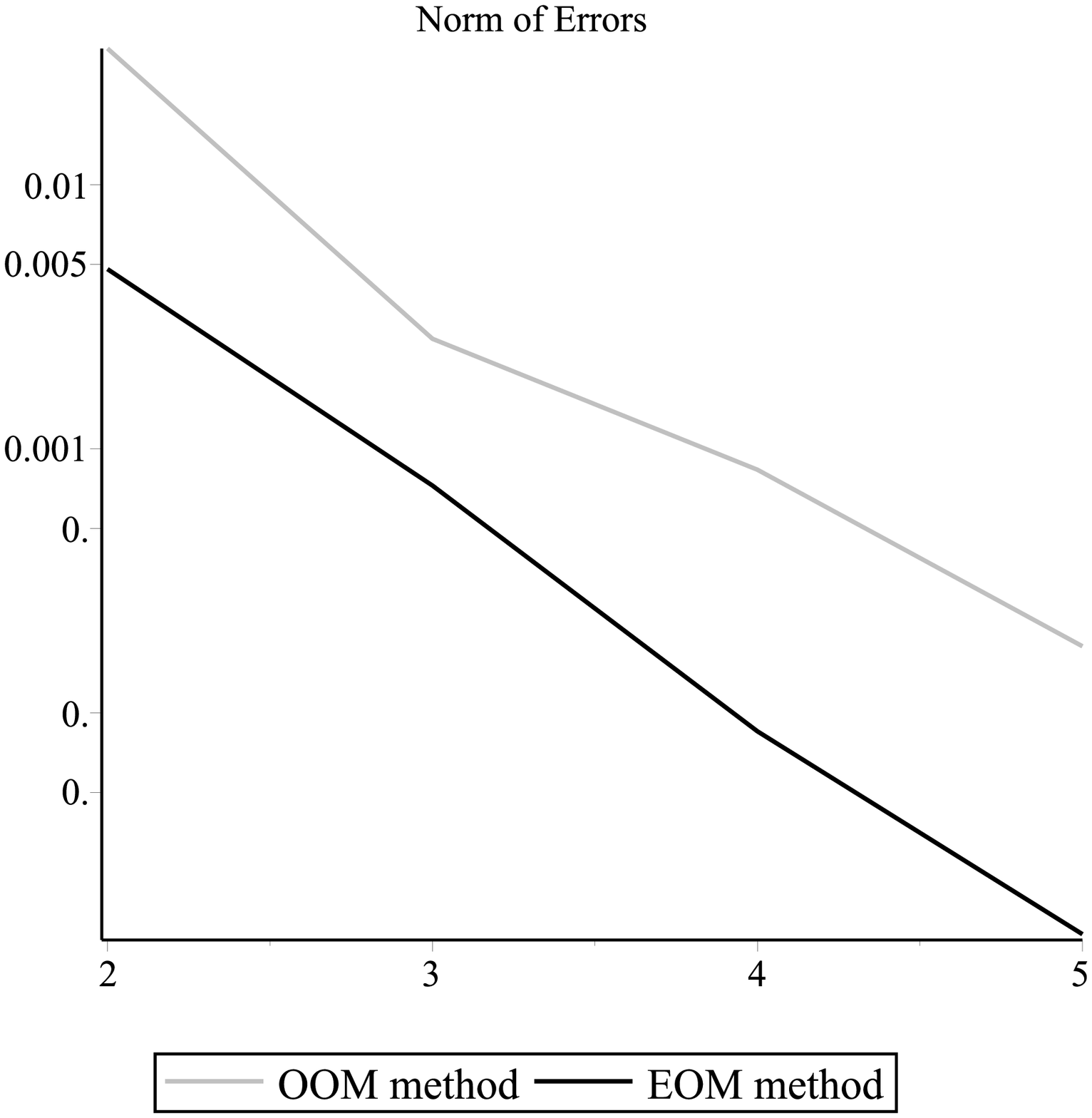}
\label{Fig_sinh_5_10_errs}
}
\subfigure[$\fNorm{Residual\left(y_m(x)\right)}_1$]{
\centering
\includegraphics[scale=0.4]{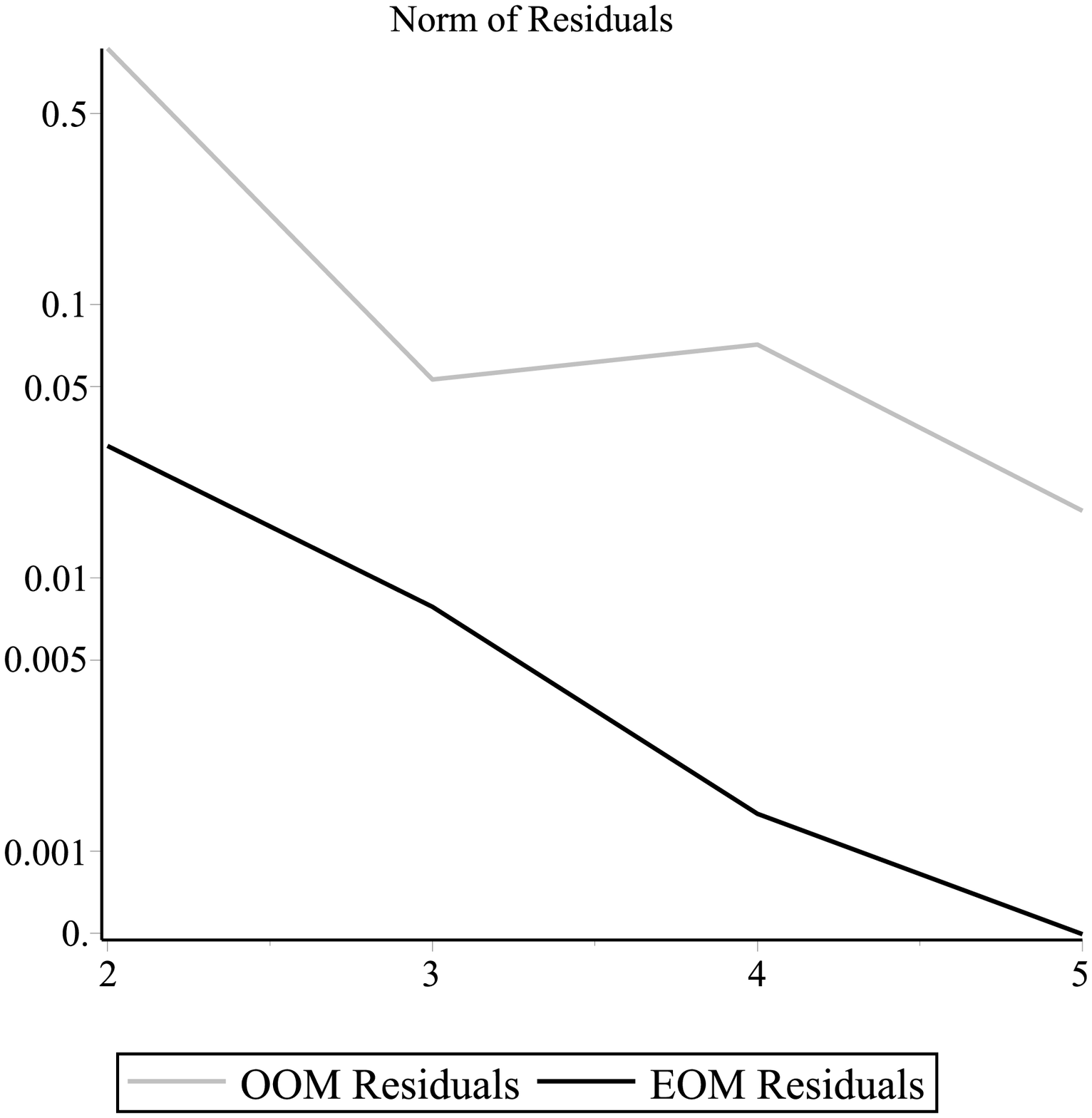}
\label{Fig_sinh_5_10_ress}
}
\caption{
the norm1 of the residual and error function plots for several $m$ values and $f(x)= 1,~g(x)= \sinh(y(x)),~N=10$}
\label{Fig_sinh_5_10_s}
\end{figure}
\begin{figure}
\centering
\subfigure[$\fNorm{Error\left(y_M(x)\right)}_1$]{
\centering
\includegraphics[scale=0.4]{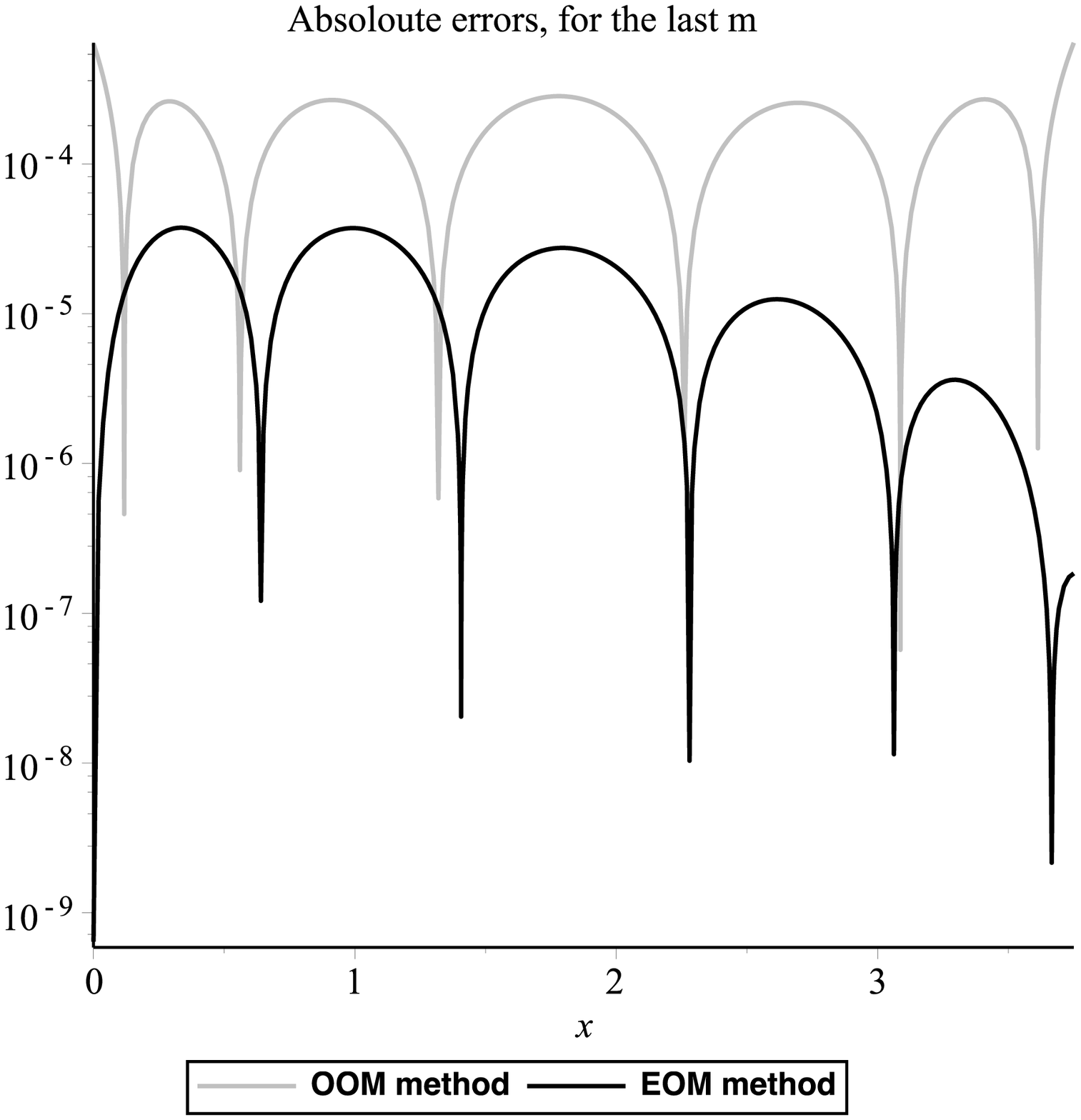}
\label{Fig_sinh_5_10_last_err}
}
\subfigure[$\fNorm{Residual\left(y_M(x)\right)}_1$]{
\centering
\includegraphics[scale=0.4]{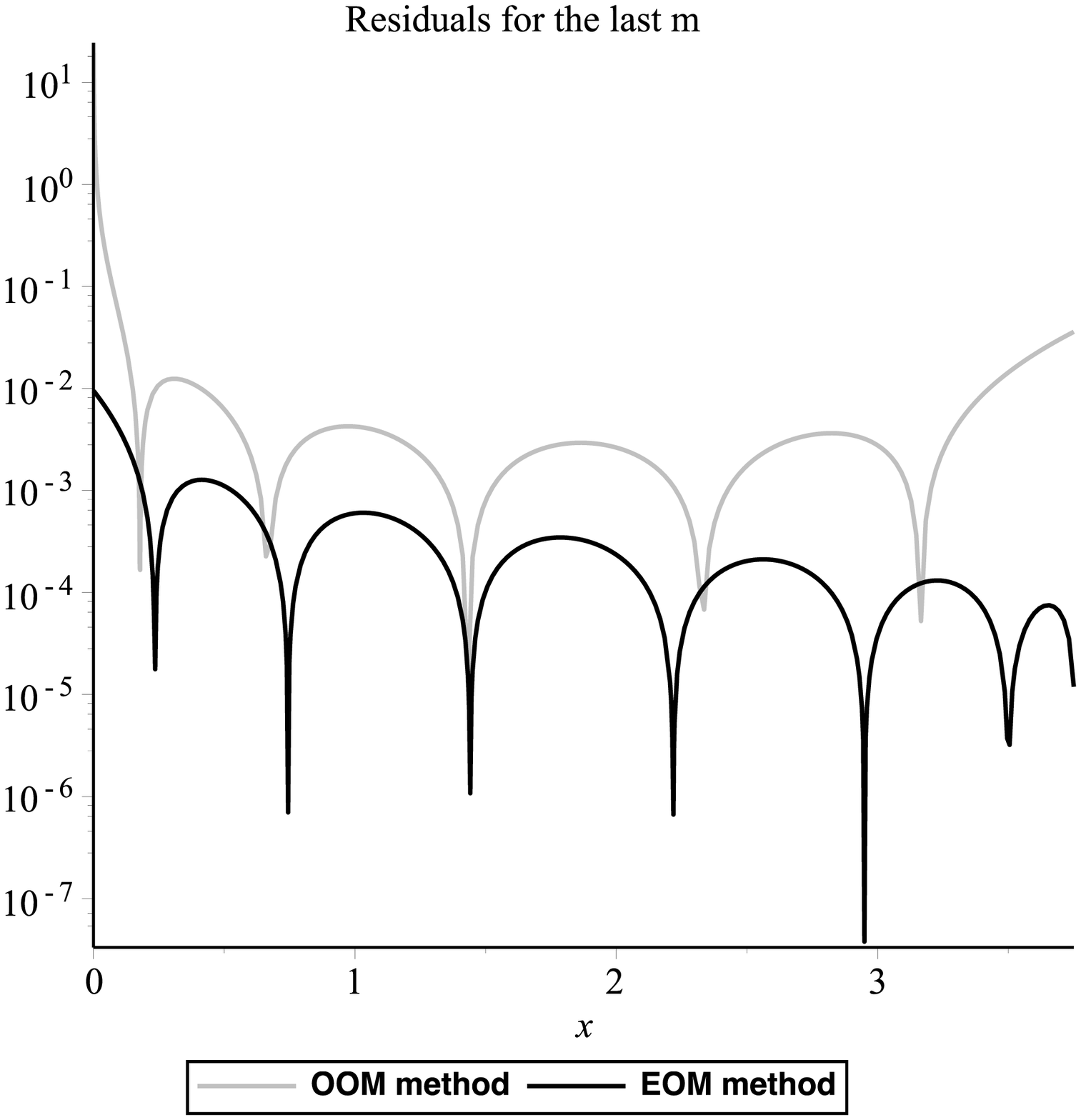}
\label{Fig_sinh_5_10_last_res}
}
\caption{
the norm1 of the residual and error function plots for the largest $m$ value ($E$) of the figure
(\ref{Fig_sinh_5_10_s})
for
$f(x)= 1,~g(x)= \sinh(y(x)),~N=10$}
\label{Fig_sinh_5_10_last}
\end{figure}
\begin{figure}
\centering
\subfigure[$\fNorm{Error\left(y_m(x)\right)}_1$]{
\centering
\includegraphics[scale=0.4]{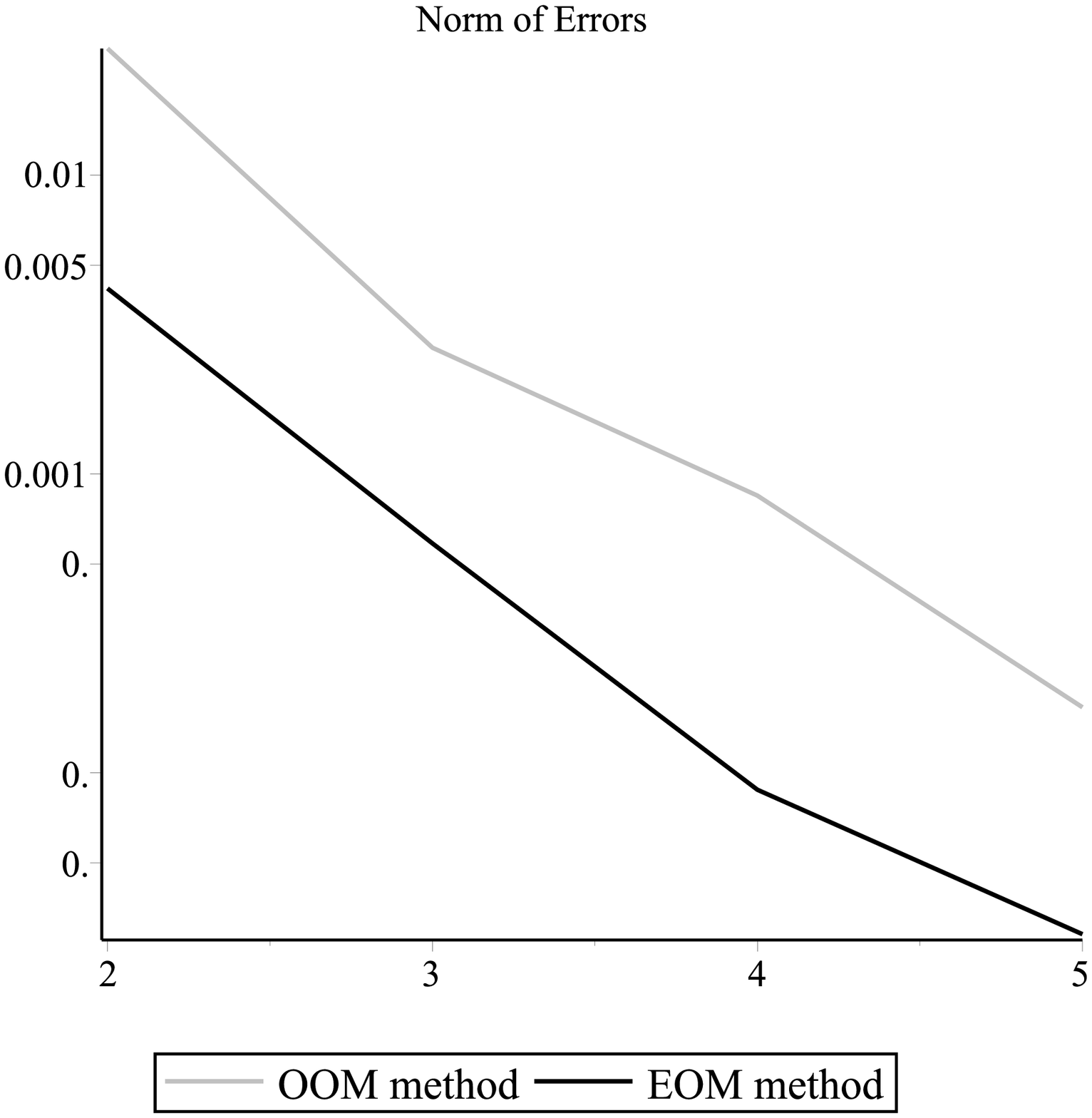}
\label{Fig_sin_5_12_errs}
}
\subfigure[$\fNorm{Residual\left(y_m(x)\right)}_1$]{
\centering
\includegraphics[scale=0.4]{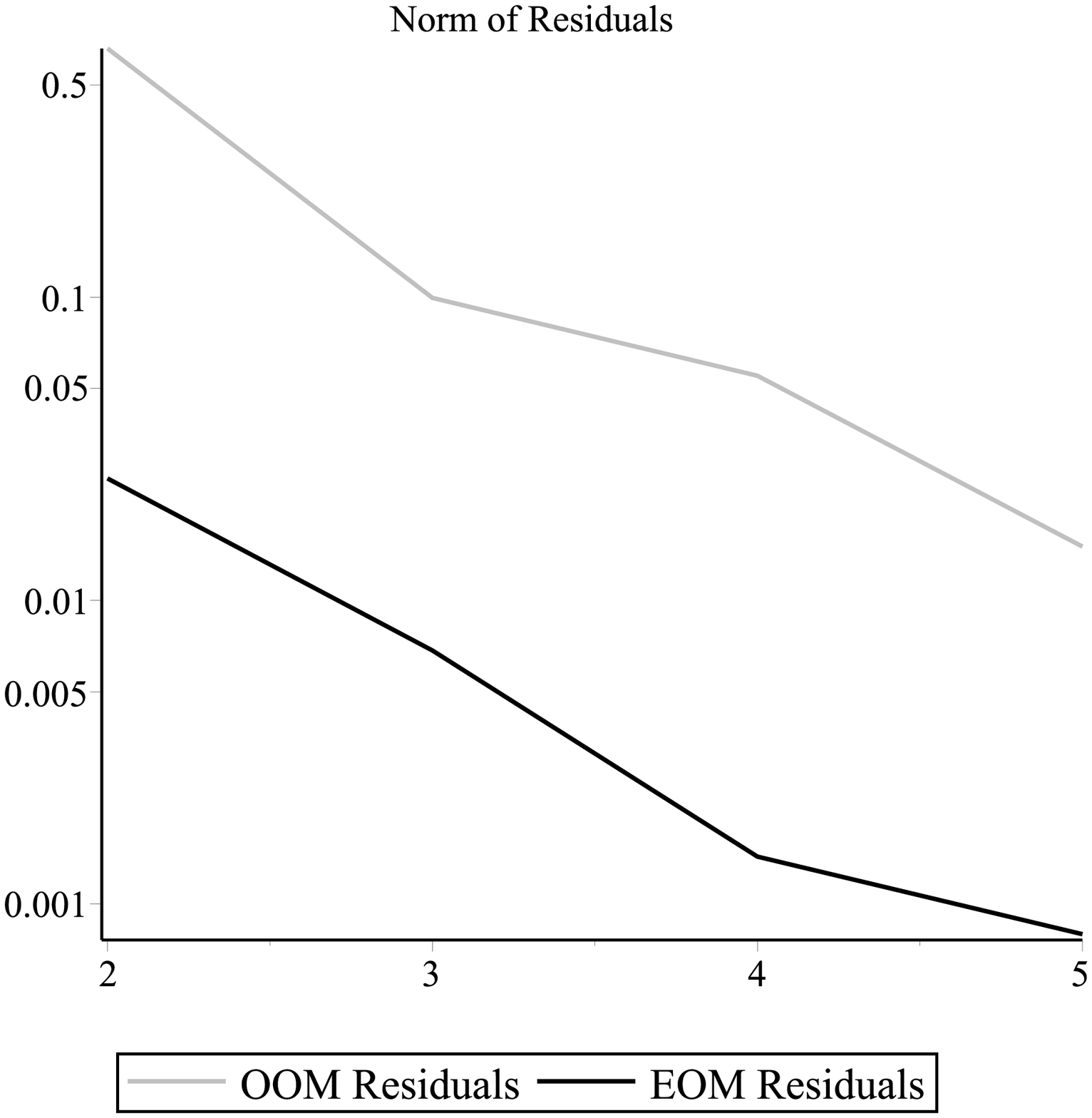}
\label{Fig_sin_5_12_ress}
}
\caption{
the norm1 of the residual and error function plots for several $m$ values and $f(x)= 1,~g(x)= \sin(y(x)),~N=12$}
\label{Fig_sin_5_12_s}
\end{figure}
\begin{figure}
\centering
\subfigure[$\fNorm{Error\left(y_M(x)\right)}_1$]{
\centering
\includegraphics[scale=0.4]{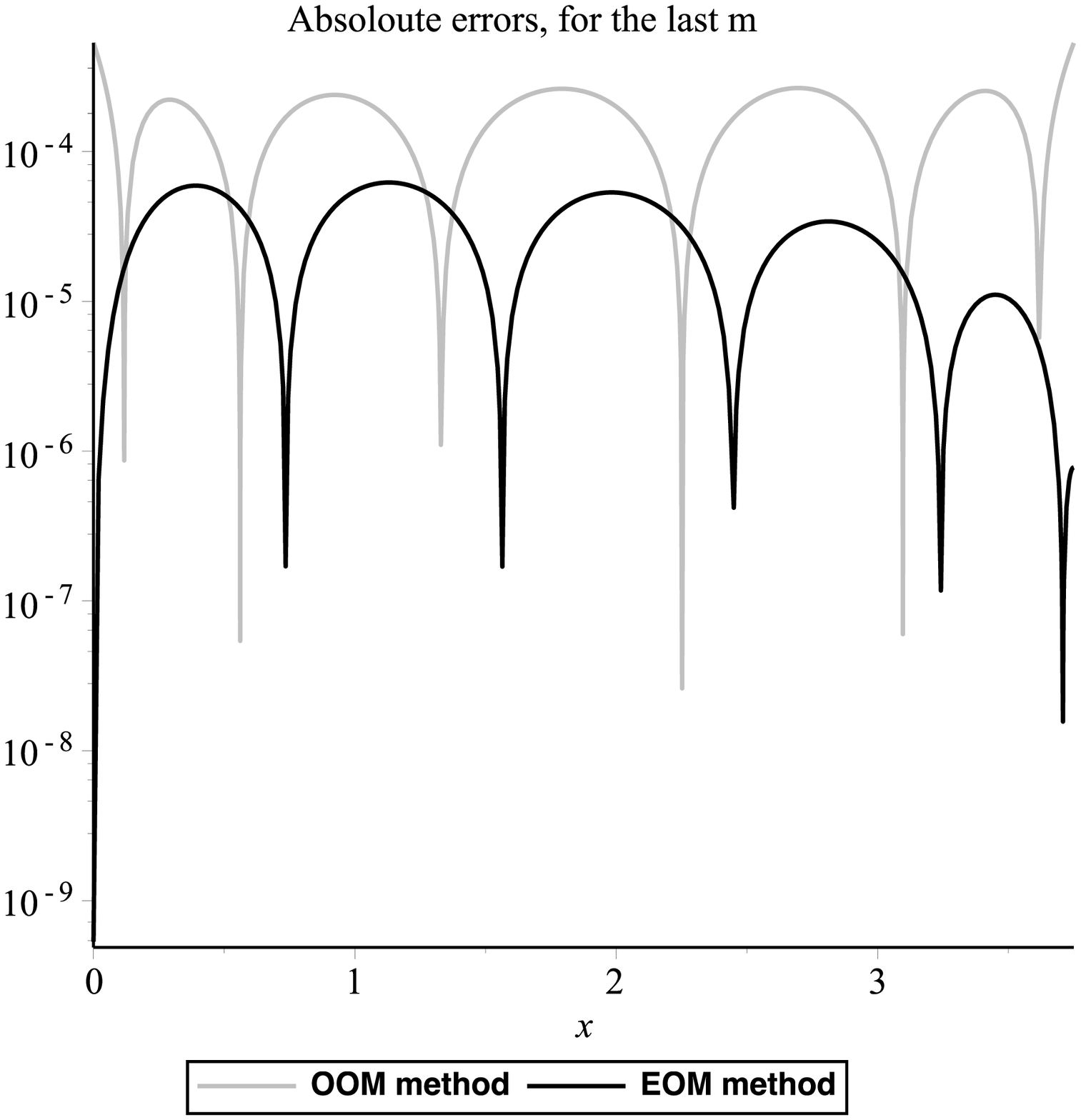}
\label{Fig_sin_5_12_last_err}
}
\subfigure[$\fNorm{Residual\left(y_M(x)\right)}_1$]{
\centering
\includegraphics[scale=0.4]{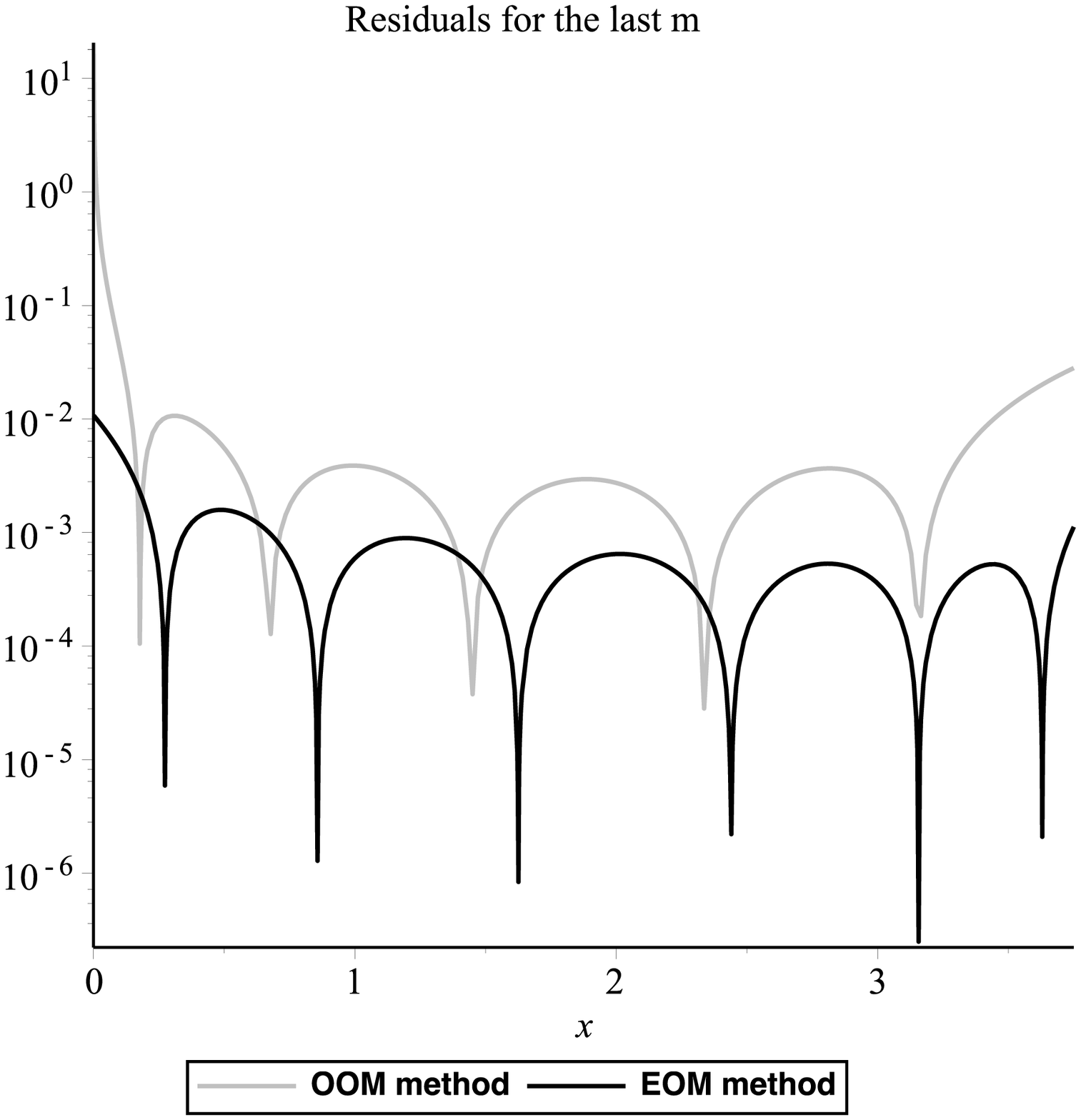}
\label{Fig_sin_5_12_last_res}
}
\caption{
the norm1 of the residual and error function plots for the largest $m$ value ($E$) of the figure
(\ref{Fig_sin_5_12_s}) for $f(x)= 1,~g(x)= \sin(y(x)),~N=12$}
\label{Fig_sin_5_12_last}
\end{figure}
\section{Concluding Remarks and suggestions for further study}\label{S_conclusion}
In this paper, we implemented the recently introduced exact operational matrices (EOMs) \cite{hossayn.tjmcs} of the Bernstein polynomials to solve the well-known Emden-Fowler equations.
To do so, well, it was necessary to achieve a relation for the best coefficients $e_i$ to approximate a function $g(x)$ in the form of $\sum_{i=1}^N{e_ix^i}$.
Then, we introduced a new matrix named by "series operational matrix" for finding the coefficients $s_i$ to write the function $g\left(y(x)\right)$ in the form of $\sum_{i=0}^{m\cdot N}{s_iB_{i,m\cdot N}(x)}$ where $B_{i,m\cdot N}(x)$s are the Bernstein polynomials of degree $m\cdot N$.\\
The differential equations which EOM idea were applied to solve them, are either linear or low-order nonlinear problems;
so, the reported results \cite{hossayn.tjmcs} were not comprehensive criteria for the superiority of the new method; therefore, we chose Emden-Fowler type equations, which have a high-order nonlinearity (because of the presence of the $g\left(y(x)\right)\simeq \sum_{i=0}^{m\cdot N}{e_iB_{i,m\cdot N}(x)}$ approximation), to be solved by EOMs.\\
For solving the differential equations by the Galerkin method, using the EOMs, we reached an almost exact residual function ($Residual(x)$) which was never obtainable by the old "ordinary operational matrices" (OOMs) in most of the problems.
We converted the residual function to a system of algebraic equations using the Galerkin operational matrix \cite{hossayn.tjmcs} and solved them to find the unknown function $y(x)$.\\
To have an appropriate criterion for the results accuracy of those problems which does not have exact solution, we solved the problem using a seventh-eighth order continuous Runge-Kutta method as an almost exact solution, using the Maple$^\copyright$ dverk78 function.\\
To see the results error convergence to zero, we applied the Galerkin method repeatedly, each new step with a larger $\psi_{n}(x)$. Then, we reported the decreasing norm of errors in some plots.
The reported results show the superiority of EOMs over OOMs.
Using the (almost) exact solution, we reported the norm1 of both methods error. We did the same for both residual functions. Then, we compared these result pairs, reporting their proportions.\\
The Emden-Fowler problem domain is $[0,\infty)$; for solving it by the Bernstein polynomials which are defined on the domain $[0,1]$, firstly, we truncated the problem domain to $[0,M]$ and then used a mapping to change its variable and solved it.\\
As some future works, EOM idea can be implemented for other basis functions.
Also, solving differential equations which have much close solutions to the vector space made by the Bernstein polynomials and some determinate norm, can illustrate the EOM efficiency much clear.
\bibliographystyle{elsart}
\bibliography{ref}
\end{document}